\documentclass[a4paper,12pt]{amsart}

\allowdisplaybreaks

\usepackage{latexsym,amsthm,amsmath,amssymb}
\usepackage{arydshln}
\usepackage{epic,eepic,eepicemu}
\usepackage{graphicx}
\usepackage{pgfplots}
\usepackage{multicol}
\pgfplotsset{compat=newest}

\usepackage{hyperref} % set hyperlinks

\renewcommand{\em}{\it}

\newtheorem{thm}{Theorem}[section]
\newtheorem{prop}[thm]{Proposition}
\newtheorem{lemma}[thm]{Lemma}

\theoremstyle{definition}

\newtheorem{example}[thm]{Example}
\newtheorem{remark}[thm]{Remark}

\newcommand{\cD}{\mathcal{D}}

\renewcommand{\S}{\mathbb{S}}

\newcommand{\R}{\mathbb{R}}
\newcommand{\N}{\mathbb{N}}

\renewcommand{\phi}{\varphi}
\renewcommand{\epsilon}{\varepsilon}

\newcommand{\tX}{\widetilde{X}}

\newcommand{\tv}{\tilde{v}}
\newcommand{\tg}{\tilde{g}}
\newcommand{\tp}{\tilde{p}}

\newcommand{\tbeta}{\tilde{\beta}}
\newcommand{\tgamma}{\tilde{\gamma}}
\newcommand{\tlambda}{\tilde{\lambda}}
\newcommand{\teta}{\tilde{\eta}}
\newcommand{\tB}{\widetilde{B}}

\newcommand{\tP}{\widetilde{P}}
\newcommand{\tU}{\widetilde{U}}

\newcommand{\tC}{\widetilde{C}}
\newcommand{\tG}{\widetilde{G}}
\newcommand{\tA}{\widetilde{A}}

\renewcommand{\tt}{\tilde{t}}

\newcommand{\hlambda}{\hat{\lambda}}
\newcommand{\hw}{\widehat{w}}
\newcommand{\hM}{\widehat{M}}

\newcommand{\hU}{\widehat{U}}

\newcommand{\vecbf}[1]{\mbox{\boldmath $#1$}}
\newcommand{\vo}{\vecbf{0}}

\newcommand{\rank}{\mathop\mathrm{rank}}

\newcommand{\tr}{\mathop\mathrm{tr}}

\title[Analytic Formulas for Alternating Projection Sequences]{Analytic Formulas for Alternating Projection Sequences for the Positive Semidefinite Cone and an Application to Convergence Analysis}

\author[H. Ochiai]{Hiroyuki Ochiai}
\address[H. Ochiai, H. Waki]{Institute of Mathematics for Industry, Kyushu University, 744 Motooka, Nishi-ku, Fukuoka 819-0395, Japan}
\email{ochiai@imi.kyushu-u.ac.jp, waki@imi.kyushu-u.ac.jp}

\author[Y. Sekiguchi]{Yoshiyuki Sekiguchi}
\address[Y. Sekiguchi]{Graduate School of Marine Science and Technology, Etchujima 2-1-8, Koto-ku, Tokyo 135-8533, Japan}
\email[Corresponding author]{yoshi-s@kaiyodai.ac.jp}

\author[H. Waki]{Hayato Waki}

\begin{document}

\begin{abstract}
We derive analytic formulas for the alternating projection method applied to the cone $\S^n_+$ of positive semidefinite matrices and an affine subspace.
More precisely, we find recursive relations on parameters representing a sequence constructed by the alternating projection method.
By applying these formulas, we analyze the alternating projection method in detail
and show that the upper bound given by the singularity degree is actually tight
when the alternating projection method is applied to $\S^3_+$ and a $3$-plane 
whose intersection is a singleton with singularity degree $2$.
\end{abstract}

\maketitle

%\tableofcontents

\section{Introduction}

\subsection{The alternating projection method}

Let $\S^n, \S^n_+$ be the sets of $n\times n$ symmetric matrices and positive semidefinite matrices, respectively.
For an affine subspace $E$ of $\S^n$, $E\cap \S^n_+$ represents the feasible region of a semidefinite programming problem. 
Thus, it is important to find a point in $E\cap \S^n_+$ 
in numerous applications across a wide range of areas \cite{C1996, D, MTW}. 
The alternating projection method for $E$ and $\S^n_+$ constructs a sequence $\{U_k\}$ with $U_0\in E$ by
\[
 U_{k+1} = P_E\circ P_{\S^n_+}(U_k),
\]
where $P_E$ and $P_{\S^n_+}$ are the projections onto $E$ and $\S^n_+$, respectively.
We call $\{U_k\}$ an \textit{AP sequence} for short.
It is known that $\{U_k\}$ converges to a point in $E\cap \S^n_+$ if $E\cap \S^n_+$ is nonempty, or to a point with displacement if $E\cap \S^n_+$ is empty; see, e.g., \cite{BB1996}, \cite{DLW2017} and the references therein. %The alternating projection method is used to find a feasible point of semidefinite programming, whose feasible region is the intersection of an affine space and $\S^n_+$. %Thus it has numerous applications across a wide range of areas \cite{C1996, D, MTW}. 

The behavior of an AP sequence has been analyzed in most  studies using inequalities related to the projections, and only upper bounds for the convergence rate are given. 
%check
In particular, \cite{DLW2017} showed that an upper bound is given by
%$O(k^{-\frac{1}{2^{d+1}-2}})$,
the singularity degree of $E\cap \S^n_+$.
%the exponent of a H\"older error bound of two semialgebraic sets.
%Since Sturm has shown that the exponent is estimated by the singularity degree of $E\cap \S^n_+$, an upper bound is given by the singularity degree \cite{}.
However, as discussed in an open question proposed in \cite{BLY2014}, 
known upper bounds might not be tight and thus it would be interesting to find a tight upper bound.
In fact, we construct an affine subspace $E$ in Example $\ref{ex:half}$, where the singularity degree of $E\cap \S^3_+$ is $2$ and the tight upper bound for the convergence rate of the AP sequence is $O(k^{-1/2})$, 
although the upper bound given by the singularity degree is $O(k^{-1/6})$.
In examples in Section $\ref{sec:eigen}$ and \cite[Example 5.2, 5.4, 5.6]{BLY2014}, the gaps between known upper bounds and the actual convergence rates are found by directly analyzing defining equations for AP sequences, instead of the inequalities related to the projections.

The purpose of this paper is to shed new light on the recursive relation defining AP sequences for $\S^n_+$ and an affine subspace with the aim of convergence analysis.
It is observed that AP sequence for $E$ and $\S^n_+$ is defined with simple projections
and it can be parameterized with respect to a basis for $E-U_*$ for $U_*\in E\cap \S^n_+$.
Thus we obtain a parametric representation for the projections that is suitable for direct calculation in the convergence analysis.

%Thus we transform these projections to fit direct calculation for convergence analysis and find a recursive relation on the parameters.

\subsection{Contributions}
It is usually a hard problem to obtain an exact convergence rate for a sequence generated by an iterative method for an optimization problem. 
We need to find a recursive relation that is explicit and appropriate for detailed computations.
For this purpose, we mainly consider the case where $E\cap \S^n_+$ is a singleton as assumed in examples in \cite{BLY2014}, to obtain exact convergence rates. 
For a non-singleton case, we present only Example $\ref{ex:posdim}$ as an application of the first formula to a local analysis around a point in the intersection. % in the case that the intersection of $\S^n_+$ and an affine subspace has a positive dimension.
It is future work to investigate the case where $E\cap \S^n_+$ is not a singleton in detail.
Under this assumption we obtain the following formulas.
After discussing the first formula, we concentrate on the case of $\S^3_+$. 
In low dimensions, we can analyze the alternating projection method in significantly more detail than in higher dimensions, thereby deepening our understanding of the method.
We hope that these results also grasp 
 the essential nature of a higher dimensional problem that is generic.

\subsubsection*{Eigenvalue formula}
For a general affine subspace $E$ and $\S^n_+$, we obtain the first analytic formula for the parameters for an AP sequence by using eigenvalues (Proposition $\ref{prop:formula_eigen}$).
In general, the eigenvalues of a parametric matrix is not readily available, and hence
this formula is not easily applied to convergence analysis. However, in some special cases, the formula has useful applications, such as constructing interesting examples (Example $\ref{ex:half}, $\ref{ex:posdim}$, \ref{ex:linear}$), and estimating convergence rates when $E$ is a line \cite{OSW2024-2}.

\subsubsection*{Analytic formula when $P_{\S^3_+}(U_k)$ is rank $1$}
We consider $\S^3_+$ and a $3$-plane $E$ for 
 simplicity. 
By numerical experiments, we see that $P_{\S^3_+}(U_k)$ is often rank $1$, and this case appears to be crucial for the convergence analysis in Section $\ref{sec:rate}$.
Thus, we additionally assume that $P_{\S^3_+}(U_k)$ is rank $1$.
Then we obtain the second analytic formula (Theorem $\ref{thm:analytic_det}$).
This formula allows us to construct a curve such that 
an AP sequence converges most slowly if the initial point is taken from the curve. 
Finding such an initial point is crucial for showing the tightness of an upper bound.

\subsubsection*{Rational formula}
We find a parameterization of the family of $3$-planes $E$ such that $E\cap \S^3_+$ is a singleton (Proposition $\ref{prop:3-planes})$. Then the set of such planes with singularity degree $2$ is fully characterized. With this characterization, we find a rational formula (the matrix $(\ref{eq:slowest})$) for the curve  giving the slowest convergence rate.
By using this formula, 
we obtain explicit expressions of $P_{\S^3_+}(U_k)$ (Theorem $\ref{thm:ratPS}$) and $P_E\circ P_{\S^3_+}(U_k)$ (Theorem $\ref{thm:ratPEPS}$), and then 
we show that the upper bound given by 
 the singularity degree is actually tight (Theorem $\ref{thm:rate})$.

\subsection{Organization}
Section $\ref{sec:prelim}$ provides the basic notation. 
The first analytic formula is obtained in Section $\ref{sec:eigen}$. 
Section $\ref{sec:analytic_det}$ contains the second analytic formula. 
In Section $\ref{sec:3-planes}$, we parameterize the set of $3$-planes whose intersections with $\S^3_+$ are a singleton. 
The rational formula for the curve giving the slowest convergence rate is given in Section $\ref{sec:rat}$. 
Section $\ref{sec:rate}$ deals with the case where the upper bound given by the singularity degree is tight.
%in the case that $E$ is a $3$-plane and $E\cap \S^n_+$ is a singleton and has singularity degree $2$.

\section{Preliminaries}
\label{sec:prelim}

%\subsection{Notation and Definitions}
%Let $\S^n, \S^n_+$ be the sets of $n\times n$ symmetric matrices, positive semidefinite matrices.
Let $[n] = \{1,\ldots,n\}$, $\langle A,B\rangle = \tr A^TB = \sum_{i,j=1}^n A_{ij}B_{ij}$,
$\|A\|_F = \sqrt{\langle A,A\rangle}$ and $\|A\|_2$ be the spectral norm of $A$.
If there is no confusion, we simply use $\|A\|$ to mean $\|A\|_F$.
The distance $d(A,E)$ from a matrix $A\in \S^n$ to a set $E\subset \S^n$ is 
defined by $d(A,E) = \inf_{X\in E}\|A - X\|_F$.
If $E$ is a closed convex subset of $\S^n$, then 
there exists a unique optimal solution to $\min_{X\in E}\|A - X\|_F$, and 
the optimal solution is called the projection of $A$ onto $E$ and denoted by $P_E(A)$.

For $f, g:\R \to \R$, we write $f(x) = O(g(x))$ as $x\to \infty$
if there exist $C, M>0$ such that
$|f(x)| \leq C g(x)$ for all $x$ with $|x| > M$.
We also write $f(x) = \Theta(g(x))$ as $x\to \infty$
if there exist $C_1, C_2 > 0$ such that
$C_1 g(x) \leq f(x) \leq C_2 g(x)$ for all $x$ with $|x| > M$.
The meaning of the statement $f(x) = O(g(x))$ as $x\to 0$ is defined similarly.
If there is no ambiguity, we simply write $f(x) = O(g(x))$, or $f(x) = \Theta(g(x))$.
For $F:\R \to \R^n$, we also write $F(x) = O(g(x))$ if $F_i(x) = O(g(x))$ for $i\in [n]$. Similarly, $F(x) = \Theta(g(x))$ is defined. 

%For a sequence $\{u_k\}\subset \R^n$ with $\bu = \lim_{k\to \infty}u_k$, we say that  $\{u_k\}$ converges 
%in the rate $O(g(k))$
%if $\|u_k - \bu\| = O(g(k))$ as $k\to \infty$, and 
%in the \textit{exact rate}
%$\Theta(g(k))$
%if $\|u_k - \bu\| = \Theta(g(k))$ as $k\to \infty$.

\section{Analytic formula for a general case}
%\subsection{General case}
\label{sec:eigen}
\subsection{Eigenvalue formula}
Suppose that $E$ is an affine subspace of $\S^n$ 
and $U_*\in E\cap \S^n_+$.
Let $B_1,\ldots, B_m$ be an orthogonal basis for $E-U_*:=\{U-U_*\in \S^n: U\in E\}$, and define
\[
\phi_0(p) = \sum_{i =1}^m p_i B_i,\  \phi(p) = U_* + \phi_0(p).
\]  
The following proposition gives the first analytic formula for the alternating projections $P_E\circ P_{\S^n_+}$.
Note that $\phi:\R^m\to E$ is bijective and thus $\phi$ has the inverse map $\phi^{-1}$.
\begin{prop}
\label{prop:formula_eigen}
Let $\tp = \phi^{-1}\circ P_E\circ P_{\S^n_+} \circ \phi(p)$ and $\lambda_1(p), \ldots, \lambda_n(p)$ be eigenvalues of $\phi(p)$. Then we have
 \[
 \tp_i = p_i - \frac{1}{\|B_i\|^2}\frac{\partial}{\partial p_i}
 \sum_{\ell \in n(p)}\frac{1}{2}\lambda_\ell^2(p),\quad i\in [m],
 \]
where $n(p) = \{\ell \in [n]:\lambda_\ell(p) < 0\}$.
\end{prop}
\begin{proof}
If $\phi(p)\in \S^n_+$, then $n(p)=\emptyset$ and hence the formula obviously holds.
Thus we assume $\phi(p)\notin \S^n_+$.
 Let $U = \phi(p)$, $V = P_{\S^n_+}(U)$ and $\tU = P_E(V)$.
By \cite[Thm 4.8]{F}, we have that $d^2(U, \S^n_+)$ is continuously differentiable and 
\[
 \nabla \frac{1}{2}d^2(U, \S_+^n) = U - P_{\S_+^n}(U),
\]
where $\nabla$ corresponds to differentiation with respect to each component of $U$.
 Since $P_E$ is the orthogonal projection onto $E$, we can easily show that
\[
  P_E(V) = U_* + \sum_i\frac{\langle B_i,V - U_*\rangle}{\|B_i\|^2}B_i
  \\
 = \sum_i\frac{\langle B_i,V \rangle}{\|B_i\|^2}B_i - \sum_i\frac{\langle B_i,U_* \rangle}{\|B_i\|^2}B_i + U_*.
\]
Thus we obtain
\begin{align*}
 \tU & = P_E(V) = P_E\left(U - \nabla \frac{1}{2}d^2(U, \S_+^n)\right)\\
&  = \sum_i\frac{\langle B_i,U - \nabla \frac{1}{2}d^2(U, \S_+^n) \rangle}{\|B_i\|^2}B_i
 - \sum_i\frac{\langle B_i,U_* \rangle}{\|B_i\|^2}B_i + U_*\\
& = P_E(U) - \sum_i\frac{\langle B_i,\nabla \frac{1}{2}d^2(U, \S_+^n)\rangle}{\|B_i\|^2}B_i
 = U - \sum_i\frac{1}{2\|B_i\|^2}\frac{\partial}{\partial p_i}d^2(\phi(p), \S_+^n)B_i.
\end{align*}
Note that the last equality follows the chain rule
\[
\frac{\partial}{\partial p_i}d^2(\phi(p), \S_+^n)
=\frac{\partial}{\partial p_i}d^2(U_* + \textstyle{\sum_{j}p_kB_k}, \S_+^n)
=\langle \nabla d^2(U, \S_+^n), B_i\rangle.
\]
Here, we see that
\[
 d^2(\phi(p), \S^n_+) = 
\|\phi(p) - P_{\S^n_+}(\phi(p))\|^2
= \sum_{\ell \in n(p)}\lambda_\ell^2(p),
\]
and thus we have
$
\tU = U - \sum_i \frac{1}{2\|B_i\|^2}\frac{\partial}{\partial p_i}\sum_{\ell \in n(p)}\lambda_\ell^2(p)B_i.
$
Since $\tU = \phi(\tp) = U_* + \sum_i \tp_iB_i$
and $U = \phi(p) = U_* + \sum_i p_iB_i$, 
we obtain the desired equality by comparing the coefficients
of $B_i$ on the both side.
\end{proof}

\subsection{Applications of the eigenvalue formula}
Since the computation of the eigenvalues of a parameterized matrix is usually difficult,
this formula is not so useful to analyze a general AP sequence.
However, this formula can be used to investigate a simpler case \cite{OSW2024-2}, or
to construct examples with interesting properties as below.

\begin{example}[Known upper bounds and actual convergence rates]
\label{ex:half}
It is well known that an upper bound of the convergence rate of alternating projections for
an affine subspace and $\S^n_+$
% \[
% E=\{X\in \S^3:\langle A_1,X \rangle = 1,\ \langle A_2, X %\rangle = 0,\ 
% \langle A_3, X \rangle = 0\}
% \]
% and $\S^3_+$ 
 is given using the singularity degree; \cite{DLW2017}. The \textit{singularity degree} is a nonnegative integer determined by the iterative process called \textit{facial reduction}. For the detail; see, e.g. \cite{BW, DW, P}. 
 We note that the singularity degree of the intersection of an affine subspace 
  and $\S^n_+$ is less than or equal to $n-1$.
  
Consider
 \begin{align*}
  E = \Big\{U \in \S^3:
 \langle 
 \left(
 \begin{smallmatrix}
 1 & 0 & 0 \\
 0  & 0 & 0 \\
 0  & 0 & 0
 \end{smallmatrix}
 \right), U
 \rangle = 1,\ 
 & \langle
 \left(
 \begin{smallmatrix}
 0  & 0 & 1 \\
 0  & 1 & 0 \\
 1  & 0 & 0
 \end{smallmatrix}
 \right), U
 \rangle = 0,\ 
 \langle
 \left(
 \begin{smallmatrix}
 0  & 1 & 0 \\
 1  & 0 & 0 \\
 0  & 0 & 0
 \end{smallmatrix}
 \right), U
 \rangle = 0,\\
 &  
 \langle
 \left(
 \begin{smallmatrix}
 0  &0  &0  \\
 0  &0  & 1 \\
 0  & 1 & 0
 \end{smallmatrix}
 \right), U
 \rangle = 0,\ 
 \langle
 \left(
 \begin{smallmatrix}
 0  & 0 & 0 \\
 0  & 0 & 0 \\
 0  & 0 & 1
 \end{smallmatrix}
 \right), U
 \rangle = 0
 \Big\}.
 \end{align*}
 Then a matrix in $E$ can be written as 
 \[
 U(t) := 
 \left(
 \begin{smallmatrix}
 1 & 0 & 0 \\
 0 & 0 & 0 \\
 0 & 0 & 0 
 \end{smallmatrix}
 \right)
 + t
 \left(
 \begin{smallmatrix}
 0 & 0 & -1\\
 0 & 2 & 0 \\
 -1 & 0 & 0 
 \end{smallmatrix}
 \right)
 \]
 and hence $E\cap \S^3_+=\{U(0)\}$.
Using the definition of singularity degree in \cite{DW},
the sigularity degree of $E\cap \S^3_+$ is $2$.
By the bound based on the singularity degree given in \cite[Theorem 2.4]{DLW2017}, an upper bound for the convergence rate of an AP sequence for $E$ and $\S^3_+$ is $O(k^{-\frac{1}{6}})$.  
However, the formula in Proposition \ref{prop:formula_eigen} ensures that the tight upper bound for the convergence rate is $O(k^{-1/2})$ as follows.

The eigenvalues of $U(t)$ is 
$\lambda_1(t) = 2t,\ 
 \lambda_2(t) = 1+t^2 - t^4 + O(t^6),\ 
 \lambda_3(t) = -t^2 + t^4 + O(t^6)$.
 Let $U(t_{k+1}) = P_E\circ P_{\S^3_+}(U(t_k))$.  If $t_0>0$ sufficiently small,
 the formula in Proposition $\ref{prop:formula_eigen}$ gives that
 \[
 t_{k+1} = t_k - \frac{1}{6}\frac{d}{dt}\frac{1}{2}\lambda_3^2(t_k)
 = t_k - \frac{1}{3}t_k^3 + O(t_k^4).
 \]
 Then \cite[Lemma 5.2]{OSW2024-2} implies that $t_k\to 0$ with $t_k>0$ and $t_k = \Theta(k^{-1/2})$.
 More precisely, $t_k \approx (3/2)^{1/2}k^{-1/2}$; see also \cite{OSW}.
 If $t_0<0$ sufficiently close to $0$, then 
 \[
 t_{k+1} = t_k - \frac{1}{6}\frac{d}{dt}\frac{1}{2}(\lambda_1^2(t_k) + \lambda_3^2(t_k))
 = t_k - \frac{2}{3}t_k - \frac{1}{3}t^3_k + O(t_k^4)
 = \frac{1}{3}t_k + O(t_k^3).
 \]
Thus $\frac{2}{3}t_k < t_{k+1} < 0$.
Hence $t_k\to 0$ with $t_k<0$ and  $t_k$ converges linearly. 
Combining the two cases, 
%Since $\|U(t) - U_*\| = \sqrt{6}|t|$, 
we see that $\|U(t_k) - U_*\| = \sqrt{6}|t_k| \leq O(k^{-1/2})$ for an arbitrary initial point.
Figure $\ref{fig:half}$ illustrates these 
 rates of convergence.
In the case that $t_0>0$, we observe from Figure $\ref{fig:half2}$ that the plot of 
$1/\|U_k - U_*\|^2$ approximately coincides with the line $33.51+0.111k$.
Hence $\|U_k - U_*\|\approx (33.51+0.111k)^{-1/2}\approx 3k^{-1/2}$. This is consistent with our estimate $\|U(t_k) - U_*\| = \sqrt{6}t_k \approx 3k^{-1/2}$.
General cases are investigated in \cite[Section 5]{OSW2024-2}.

 %%%%%%%%%%%%%%%%%%%%%%%%%%%%%%%%%%%%%%%%%%%%
\begin{figure}[ht]
\begin{tabular}{cc}
        \begin{minipage}{.5\textwidth}
            \centering
\begin{tikzpicture}[scale=0.73]

\begin{axis}[grid=major, xlabel={$k$}, legend entries={$\sqrt{k}\|U_k-U_*\|$, $\sqrt[6]{k}\|U_k-U_*\|$}, legend style={at={(axis cs:50000,2)}}]
\addplot[red] table [x=k, y=1norm2,] {data32_p.dat};
%\addplot[blue, dotted, very thick] table [x=k, y=e+c*k,] {data32_p.dat};
\addplot[blue] table [x=k, y=1norm6,] {data32_p.dat};
%\addplot[black, dotted, very thick] table [x=k, y=line2,] {data32_p.dat};
\end{axis}

\end{tikzpicture}
%\caption{Plot of $\|U_k-U^*\|$ with $t_0 > 0$ in Example 3.2}
        \end{minipage}
        \begin{minipage}{.5\textwidth}
\centering
\begin{tikzpicture}[scale=0.73]
\begin{axis}[grid=major, xlabel={$k$},ymode=log, legend entries={$\|U_{k}-U_*\|$, $0.070\times (0.3333)^k$}, legend style={at={(axis cs:19.5,10e-13)}}]
\addplot[red] table [x=k, y=norm] {data32_m.dat};
\addplot[blue, very thick, dotted] table [x=k, y=Ack] {data32_m.dat};
\end{axis}
\end{tikzpicture}
        \end{minipage}                
    \end{tabular}      
\caption{The left figure displays a plot of $\|U_k - U_*\|$ with $t_0>0$ and the right figure displays a plot with $t_0<0$ in Example $\ref{ex:half}$, and the line fitting for the plot in the right figure.} 
\label{fig:half}  
\end{figure}
%%%%%%%%%%%%%%%%%%%%%%%%%%%%%%%%%%%%%%%%%%%%

\begin{figure}[ht]
            \centering
\begin{tikzpicture}[scale=0.73]
\begin{axis}[grid=major, xlabel={$k$}, legend entries={$1/\|U_{k}-U_*\|^2$, $33.51 +0.1111k$}, legend style={at={(axis cs:7000,200)}}, 
]
\addplot[red] table [x=k, y=1norm2,] {data32_p2.dat};
\addplot[blue, very thick, dotted] table [x=k, y=e+c*k,] {data32_p2.dat};
\end{axis}
\end{tikzpicture}
\caption{Plot of $1/\|U_k-U_*\|^2$ with $t_0 > 0$ in Example 3.2 and the line fitting.}
\label{fig:half2}  
\end{figure}

%%%%%%%%%%%%%%%%%%%%%%%%%%%%%%%%%%%%%%%%%%
\end{example}

\begin{example}[Positive dimensional intersection]
\label{ex:posdim}
Proposition $\ref{prop:formula_eigen}$ can be used in the case that an intersection has a positive dimension.
Consider
 \begin{align*}
  E = \Big\{U \in \S^3:
 \langle 
 \left(
 \begin{smallmatrix}
 1 & 0 & 0 \\
 0  & 1 & 0 \\
 0  & 0 & 0
 \end{smallmatrix}
 \right), U
 \rangle = 1,\ 
 & \langle
 \left(
 \begin{smallmatrix}
 0  & 0 & 0 \\
 0  & 1 & 0 \\
 0  & 0 & -1
 \end{smallmatrix}
 \right), U
 \rangle = 0,\ 
 \langle
 \left(
 \begin{smallmatrix}
 0  & 0 & 0 \\
 0  & 1 & -1 \\
 0  & -1 & 1
 \end{smallmatrix}
 \right), U
 \rangle = 0,\\
 &  
 \langle
 \left(
 \begin{smallmatrix}
 0  &1  &0  \\
 1  &0  & 0 \\
 0  & 0 & 0
 \end{smallmatrix}
 \right), U
 \rangle = 0,\ 
 \langle
 \left(
 \begin{smallmatrix}
 0  & 0 & 1 \\
 0  & 0 & 0 \\
 1  & 0 & 0
 \end{smallmatrix}
 \right), U
 \rangle = 0
 \Big\}.
 \end{align*}
Then a matrix in $E$ can be written as
\[
U(t)=
\left(
\begin{smallmatrix}
1 & 0 & 0 \\
0 & 0 & 0 \\
0 & 0 & 0
\end{smallmatrix}
\right)
+ t
\left(
\begin{smallmatrix}
-1 & 0 & 0 \\
0 & 1 & 1 \\
0 & 1 & 1
\end{smallmatrix}
\right),
\]
and hence $E\cap \S^3_+ = \{U(t): 0 \leq t \leq 1\}$.
The singularity degree of $E\cap \S^3_+$ is $2$.
Now the eigenvalues of $U(t)$ are $\lambda_1(t) = 1 - t,\ \lambda_2(t) = 2t,\ \lambda_3(t) = 0$.
Consider an AP sequence $U(t_{k+1}) = P_E\circ P_{\S^3_+}(U(t_k))$.

If we take the initial point as $U(t_0)$ with $t_0<0$, then the AP sequence is expected to converge to $U(0)$.
By using the formula in Proposition $\ref{prop:formula_eigen}$ with $U_* = U(0)$, we have
\[
t_{k+1} = t_k - \frac{1}{5}\frac{d}{dt}\frac{1}{2}\lambda_2^2(t_k) = t_k - \frac{4}{5}t_k = \frac{1}{5}t_k.
\]
Thus $t_k\to 0$ with $t_k<0$, and in fact $U(t_k)\to U(0)$ linearly.

If  we take the initial point as $U(t_0)$ with $t_0>1$, then the AP sequence is expected to converge to $U(1)$.
To use the formula in Proposition $\ref{prop:formula_eigen}$ with $U_* = U(1)$,
we define $\hU(s) = U(s+1)$ and $s_k = t_k+1$. Then the eigenvalues of $\hU(s)$ are $\hlambda_1(s) = -s,\ \hlambda_2(s) = 2(s+1),\ \hlambda_3(s) = 0$
By Proposition $\ref{prop:formula_eigen}$, we have
\[
s_{k+1} = s_k - \frac{1}{5}\frac{d}{ds}\frac{1}{2}\hlambda_1^2(s_k) = s_k - \frac{1}{5}s_k = \frac{4}{5}s_k.
\]
Thus $s_k\to 0$ with $s_k>0$ and hence $U(s_{k+1}+1)\to U(1)$ linearly. Therefore, the actual convergence rate of an AP sequence is linear.
%Figure $\ref{fig:posdim}$ displays log-plots of $\|U_k - U_*\|$ with $t_0>1$ and $t_0<0$, and their line fittings.
Figure $\ref{fig:posdim}$ is consistent with our estimates that the convergence rates are $O((4/5)^k)$ and $O((1/5)^k)$, respectively.
%%%%%%%%%%%%%%%%%%%%%%%%%%%%%%%%%%%%%%%%%%
\begin{figure}[ht]
\begin{tabular}{cc}
        \begin{minipage}{.45\textwidth}
            \centering
\begin{tikzpicture}[scale=0.7]
\begin{axis}[grid=major, xlabel={$k$},ymode=log, legend entries={$\|U_{k}-U_*\|$, $1.981\times(0.8000)^k$}, legend style={at={(axis cs:65,10e-8)}}]
\addplot[red] table [x=k, y=norm] {data33_p.dat};
\addplot[blue, very thick, dotted] table [x=k, y=Ack] {data33_p.dat};
\end{axis}
\end{tikzpicture}
%\caption{Log-plots of $\|U_k-U^*\|$ with $t_0 > 1$ and a line in Example 3.3}
        \end{minipage}
        \begin{minipage}{.45\textwidth}
\centering
\begin{tikzpicture}[scale=0.7]
\begin{axis}[grid=major, xlabel={$k$},ymode=log, legend entries={$\|U_{k}-U_*\|$, $0.174\times (0.2000)^k$}, legend style={at={(axis cs:19,10e-17)}}]
\addplot[red] table [x=k, y=norm] {data33_m.dat};
\addplot[blue, very thick, dotted] table [x=k, y=Ack] {data33_m.dat};
\end{axis}
\end{tikzpicture}
%\caption{Log-plots of $\|U_k-U^*\|$ with $t_0 < 0$ and a line in Example 3.3}
        \end{minipage}
    \end{tabular}
    \caption{The left figure displays a log-plot of $\|U_k-U_*\|$ with $t_0 > 1$ and the right figure shows a log-plot with $t_0 < 0$  in Example $\ref{ex:posdim}$, and their line fittings.}      
\label{fig:posdim}
\end{figure}
%%%%%%%%%%%%%%%%%%%%%%%%%%%%%%%%%%%%%%%%%%

\end{example}

\begin{example}[Intersection with $2$-plane]\label{ex:linear}  
Consider the parametrized matrix
{\small \[
U(p) = \begin{pmatrix}
 1 & 0 & p_1 & p_2 \\
 0 & 0 & p_1 & p_2 \\
 p_1 & p_1 & 0 & 0 \\
 p_2 & p_2 & 0 & 0
\end{pmatrix}.
 \]}%
Then $U(p)$ represents $2$-plane $E$ in $\S^4$.
%that intersects with $\S^4_+$ %nontransversely; i.e., the intersection does not include an open set. 
Now $E\cap \S^4_+ = \{U_*\}$ with $U_* = U(0,0)$, and
%$U_* = \left(
%\begin{smallmatrix}
% 1 & 0 & 0 & 0 \\
% 0 & 0 & 0 & 0 \\
% 0 & 0 & 0 & 0 \\
% 0 & 0 & 0 & 0 \\
%\end{smallmatrix}
%\right)
%$.
the singularity degree of $E\cap \S^n_+$ is $1$.
By the bounds based on the singularity degree given in \cite[Theorem 2.4]{DLW2017}, an upper bound for the convergence rate of AP sequence $\{U(p^{(k)})\}$ for $E$ and $\S^4_+$ is $O(k^{-\frac{1}{2}})$.
However, the formula in Proposition \ref{prop:formula_eigen} ensures that the actual convergence rate of the AP sequence is linear as follows.

For $r = \sqrt{p_1^2 + p_2^2}$, the characteristic equation of $U(p)$ is 
written as 
\[
 \lambda\left(\lambda^3 - \lambda^2 - 2r^2 \lambda + r^2\right) = 0.
\]
We consider the parametric equation $\lambda^3 - \lambda^2 - 2r^2 \lambda + r^2 = 0$ with the parameter $r$.
When $r = 0$, the polynomial $\lambda^3 - \lambda^2$
has $1$ as a simple zero and $0$ as a double zero.
Since the constant term is positive for $r>0$,
by considering the graph of $\lambda^3 - \lambda^2$, we see that the solutions to the equation are $1 + O(r)$, a positive and a negative solution
for sufficiently small $r>0$.
To apply Proposition $\ref{prop:formula_eigen}$, we will find the negative solution.
By putting $\lambda = -ru$, we obtain
$-r^3 u^3 - r^2 u^2 + 2r^3 u + r^2 = 0$
and thus $\displaystyle r = \frac{(u + 1)(u - 1)}{u(2 - u^2)}$.
Additionally, if we put $u = 1 + v$, then
 $\displaystyle r = \frac{v(v + 2)}{(v + 1)(1 - 2v - v^2)}$.
This means that $r$ is a rational function in $v$.
By computation, we have $dr/dv|_{v=0} = 2 \neq 0$, and
then the inverse function $v=v(r)$ near $r=0$ is also analytic.
Since $r = 2v + 3v^2 + O(v^3)$, we have $v = r/2 - 3r^2/8 + O(r^3)$. Thus the negative eigenvalue can be written as $\lambda = -r + h(r) = -\sqrt{p_1^2 + p_2^2} + h(\sqrt{p_1^2 + p_2^2})$, where $h(r) = O(r^2)$ is a convergent power series around $r = 0$. Then
\[
 \frac{\partial}{\partial p_1}\lambda^2
= 2p_1 - \frac{2p_1}{r}h(r) - 2p_1 h'(r) 
+ \frac{2p_1}{r}h(r)h'(r)
= 2p_1\left(1 + O(r)\right),
\]
and hence 
\[
 \tp_1 = p_1 - \frac{1}{8}\frac{\partial}{\partial p_1}
\lambda^2 = \frac{3}{4}p_1 + p_1 O(r).
\]
Similarly, $\tp_2 = \frac{3}{4}p_2 + p_2 O(r)$. Thus we have
\[
\tilde{r} := \sqrt{\tp_1^2 + \tp_2^2} = \sqrt{\frac{9}{16}r^2 + O(r^3)}
= \frac{3}{4}r + O(r^2).
\]
Since 
$\|U(p) - U_*\| = 2r$, we see that for small $\delta>0$,
\[
\|U(\tp) - U_*\|=2\tilde{r} = 2\cdot\frac{3}{4}r(1 + O(r)) < 2\cdot\left(\frac{3}{4}+\delta\right) r = \left(\frac{3}{4}+\delta\right)\|U(p) - U_*\|
\]
for $p$ sufficiently close to $(0,0)$.
Therefore, the actual convergence rate of an AP sequence is linear. 
Figure $\ref{fig:linear}$ is consistent with our estimate that the convergence rate of $\|U_k - U_*\|$
 is approximately $O((3/4)^k)$.
%%%%%%%%%%%%%%%%%%%%%%%%%%%%%%%%%%%%%%%%%%
\begin{figure}[htbp]
\centering
\begin{tikzpicture}[scale=0.7]
\begin{axis}[grid=major, xlabel={$k$}, ymode=log, legend entries={$\|U_{k}-U_*\|$, $0.146\times(0.749348)^k$}, legend style={at={(axis cs:31,0.03)},anchor=south east}]
\addplot[red, ] table [x=k, y=norm,] {data34.dat};
\addplot[blue, very thick, dotted] table [x=k, y=A*c^k] {data34.dat};
\end{axis}
\end{tikzpicture}
\caption{Log plot of $\|U_k-U_*\|$ in Example 3.4}
\label{fig:linear}
\end{figure}
%%%%%%%%%%%%%%%%%%%%%%%%%%%%%%%%%%%%%%%%%%

\end{example}

\section{Analytic formula when $P_{\S^3_+}(U)$ is rank $1$}
%\subsection{The rank $1$ projection}
\label{sec:analytic_det}
In the rest of the paper, we consider $\S_+^3$ 
and an affine subspace $E$ whose intersection with $\S^3_+$ is 
$U_* = \left(
\begin{smallmatrix}
1 & 0 & 0 \\
0 & 0 & 0 \\
0 & 0 & 0 
\end{smallmatrix}
\right)
$.
Let $D_k = \{V \in \S^3: \rank V \leq k\}$; i.e., the determinantal variety of rank at most $k$.
Then $D_2$ contains the boundary of $\S^3_+$ and its singular locus is $D_1$. Since $\S^3_+$ is convex, $D_1\cap \S^3_+$ is geometrically interpreted as the ridge of the boundary of $\S^3_+$. %in a geometric point of view.
%Then it is expected that $P_{\S^3_+}(U)$ is often included in $D_1$ for $U\in E$. 
Hence, $P_{\S^3_+}(U)$ is expected to be frequently included in $D_1$ for $U\in E$.
In fact, this often happens in numerical experiments.
Thus we consider the case that $P_{\S^3_+}(U)\in D_1$ for $U\in E$ sufficiently close to $U_*$. In Section $\ref{sec:rate}$, this case will appear to be crucial for the convergence analysis.
 We note that $\dim D_1 = 3$ and $\dim \S^3 = 6$ and then the complementary dimension of $D_1$
is $3$. 
Since $E\cap \S^3_+$ is a singleton and contained in $D_1$, we assume that the dimension of $E$ is $3$ so that the intersection of $E$ and $D_1$ is zero-dimensional in general.
%
%If we assume the projection onto $\S^3_+$ is a rank $1$ matrix, then we obtain the second analytic formula for an AP sequence.
%
\subsection{Analytic formula with a distance function}
Let $B_1,B_2,B_3$ be an orthogonal basis for $E-U_*$ and
\[
%U_* = 
%\begin{pmatrix}
% 1 & 0 & 0 \\
%0 & 0 & 0 \\
%0 & 0 & 0
%\end{pmatrix},\quad
  \phi(p) = U_* + p_1 B_1 + p_2 B_2 + p_3 B_3,\quad
 \psi(x) =  \frac{1}{x_1}
\begin{pmatrix}
 x_1\\
x_2\\
x_3
\end{pmatrix}
\begin{pmatrix}
 x_1 & x_2 & x_3
\end{pmatrix}
\]
and $x_*= (1,0,0)$. 
Then $\phi:\R^3\to E$ and $\psi:\{x\in \R^3:x_1\neq 0\}\to D_1$.
It is easily verified that the image of $\psi$ contains an open neighborhood of $U_*$ in $D_1$.

For $f:\R^3\to \R$, define $\partial_k f(x) = \frac{\partial}{\partial x_k}f(x)$ for $k=1,2,3$.
\begin{thm}
\label{thm:analytic_det}
Suppose that $E$ is a 3-plane in $\S^3$, 
$E\cap \S^3_+ = \{U_*\}$ and 
$P_{\S^3_+}\circ\phi(p)$ has rank $1$. 
Let $\tp,x\in \R^3$ satisfy the relation
$\displaystyle \phi(p) \overset{P_{\S^3_+}}{\longmapsto} \psi(x) 
 \overset{P_E}{\longmapsto} \phi(\tp)$.
Then we have
\[
 M(x)(\tp - p) + \nabla_x \frac{1}{2}d^2(\psi(x), E) = 0,
\]
where 
$  M(x) = \Big(\left\langle \partial_k\psi(x), B_i\right\rangle\Big)_{k,i}\in \R^{3\times 3}$.
Moreover, if $\det M(x_*)\neq 0$ and $p$ is sufficiently close to $(0,0,0)$, then
\[
  \tp = p -  M(x)^{-1}\nabla_x \frac{1}{2}d^2(\psi(x), E).
\]
\end{thm}
\begin{proof}
Since $P_{S^3_+}(\phi(p))\in D_1\subset \S^3_+$, we see that
$P_{S^3_+}(\phi(p)) = \mathop\mathrm{argmin}_{x'\in \R^3, x'_1\neq 0}\|\psi(x') - \phi(p)\|_F^2$.
In addition, $\phi(\tp) = P_E(\psi(x))$ is equivalent to $(\psi(x) - \phi(\tp))\perp E$.
Thus, we have
\[
\psi(x) = P_{\S^3_+}(\phi(p))\quad 
\Longleftrightarrow
 (*)
\begin{cases}
& \langle \partial_k \psi(x), \psi(x) - \phi(p) \rangle = 0,\ k = 1, 2, 3\\
& \tp = 
\left(
\frac{\left\langle B_1, \psi(x) - U_*\right\rangle}{\|B_1\|^2},
\frac{\left\langle B_2, \psi(x) - U_*\right\rangle}{\|B_2\|^2},
\frac{\left\langle B_3, \psi(x) - U_*\right\rangle}{\|B_3\|^2}
\right)^T.
\end{cases}
\] 
By extending the basis for $E-U_*$, we obtain
an orthogonal basis $B_1, \ldots, B_6$ for $\S^3$.
Then the distance function can be written as
\begin{align*}
 d^2(\psi(x), E) & = \|\psi(x) - P_E(\psi(x))\|^2
\\
& = \left\|
U_* + 
\sum_{j=1}^6
\frac{\left\langle B_j, \psi(x) - U_*\right\rangle}{\|B_j\|^2}B_j
- \left(U_* + 
\sum_{i=1}^3 
\frac{\left\langle B_i, \psi(x)  - U_*\right\rangle}{\|B_i\|^2}B_i
\right)
\right\|^2\\
& =\left\|
\sum_{j=4}^6
\frac{\left\langle B_j, \psi(x) - U_*\right\rangle}{\|B_j\|^2}B_j
\right\|^2 
= 
\sum_{j=4}^6
\frac{\left\langle B_j, \psi(x) - U_*\right\rangle^2}{\|B_j\|^2}.
\end{align*}
By rewriting the condition $(*)$ using the basis, we have, for $k=1, 2, 3$, 
\begin{align*}
 0 & = \langle \partial_k \psi(x),\ \psi(x) - \phi(p) \rangle \\
 & = \left\langle 
\sum_{j=1}^6\partial_k \psi_i(x),\
U_* + \sum_{j=1}^6\psi_j(x) - 
\left(U_* + \sum_{i=1}^3 \phi_i(p) \right)\right\rangle\\
& = \left\langle
\sum_{j=1}^6 \frac{\left\langle B_j, \partial_k\psi(x) \right\rangle}{\|B_j\|^2}B_j,\ 
\sum_{j=1}^6\frac{\left\langle B_j, \psi(x) - U_*\right\rangle}{\|B_j\|^2}B_j  
- \sum_{i=1}^3 p_i B_i\right\rangle \\
& = \left\langle 
\sum_{j=1}^6 \frac{\left\langle B_j, \partial_k\psi(x) \right\rangle}{\|B_j\|^2}B_j,\ 
\sum_{i=1}^3\left(\frac{\left\langle B_i, \psi(x) - U_*\right\rangle}{\|B_i\|^2} - p_i\right)B_i 
+ \sum_{j=4}^6\frac{\left\langle B_j, \psi(x) - U_*\right\rangle}{\|B_j\|^2}B_j  \right\rangle \\
& =
\sum_{i=1}^3 \left\langle B_i, \partial_k\psi(x) \right\rangle
\left(\frac{\left\langle B_i,\psi(x) - U_*\right\rangle}{\|B_i\|^2} - p_i\right)
+ \sum_{j=4}^6\left\langle B_j, \partial_k\psi(x) \right\rangle\frac{\left\langle B_j, \psi(x) - U_*\right\rangle}{\|B_j\|^2}\\
& = 
\sum_{i=1}^3 \left\langle B_i, \partial_k\psi(x) \right\rangle
(\tp_i - p_i)
+ \partial_k\cdot \frac{1}{2}\sum_{j=4}^6
\left(\frac{\left\langle B_j, \psi(x) - U_*\right\rangle}{\|B_j\|}\right)^2\\
& = 
\sum_{i=1}^3 \left\langle B_i, \partial_k\psi(x) \right\rangle
(\tp_i - p_i)
+ \partial_k\cdot \frac{1}{2}d^2(\psi(x), E).
\end{align*}

\end{proof}

\begin{remark}
%Let $D_1$ be the determinantal variety of rank less than or equal to $1$. Then $D_1$ is $3$-dimensional. Since $\dim S^3_+ = 6$, the complimentary dimension of $D_1$ is $3$. Thus by intersecting $3$-plane with $D_1$, the intersection is zero-dimensional in general. In addition, 
Since both of $D_1$ and $E$ have $3$ parameters, the matrix $M(x)$ defined in Theorem $\ref{thm:analytic_det}$ is a square matrix.
If we consider general $D_k\subset\S^n_+$ and $E$, then $M(x)$ may become a rectangle matrix.
\end{remark}

\begin{remark}
We can connect the formula in Proposition $\ref{prop:formula_eigen}$
with that in Theorem $\ref{thm:analytic_det}$ in the following way.
 Let $\tp,x\in \R^3$ be such that
 $\displaystyle \phi(p) \overset{P_{\S^3_+}}{\longmapsto} \psi(x) 
 \overset{P_E}{\longmapsto} \phi(\tp)$.
 %$\psi(x) = P_{\S^3_+}(\phi(p))$.
Since $\tp_i - p_i = -\frac{\partial}{\partial p_i} \frac{1}{2\|B_i\|^2}d^2(\phi(p), \S^3_+)$,
we have
 \[
 \nabla_x \frac{1}{2}d^2(\psi(x), E) = \hM(x)\nabla_p \frac{1}{2}d^2(\phi(p), \S^3_+), \text{ where }
 \hM(x) = \left(\frac{\left\langle \partial_k\psi(x), B_i \right\rangle}{\|B_i\|^2}\right)_{k,i}.
 \]
 Let $U = \phi(p)$. Then we can also write
 \[
 \nabla_x \frac{1}{2}d^2(P_{\S^3_+}(U), E) =  \hM(x)\nabla_p \frac{1}{2}d^2(U, \S^3_+).
 \]
%This relates $d^2(P_{\S^3_+}(U), E)$ with $d^2(U, \S^3_+)$ in a sense.
\end{remark}

%If $\det M(x_*)\neq 0$, then this formula is more useful to anlyse a general AP sequence than the eigenvalue formula.

\subsection{Equations for the slowest curve}
\label{sec:slowest}
%Suppose $E$ is Type $2$. If $c_3\neq 0$, then $\det M(x_*)\neq 0$ and thus
The known results give upper bounds for the convergence rate of an AP sequence.
However, it is hard to show that a given upper bound is actually tight.
A key to show the tightness is to obtain a candidate for the initial point with which the AP sequence converges most slowly.

%Using the fomula in Theorem $\ref{thm:analytic_det}$, we can obtain 
%In standard convergence analysis, an upper bound
If $\det M(x_*)\neq 0$, then we have from Theorem  $\ref{thm:analytic_det}$ that
\[
  \tp = p - M(x)^{-1}\nabla_x \frac{1}{2}d^2(\psi(x), E).
\] 
Thus if we find a point that is a minimizer of $\displaystyle \min_{\|x\|=\delta}\|M(x)^{-1}\nabla \frac{1}{2}d^2(\psi(x), E)\|$ for $\delta\neq 0$, then the point gives the shortest step size with respect to the parameters of the alternating projection method.

%Now $\nabla^2|_{x_*} \frac{1}{2}d^2(\psi(x), E)$ has rank $2$.
%Thus two of the rows of $\nabla \frac{1}{2}d^2(\psi(x), E)=0$ can be solved
%by an analytic curve $x(t)\in \R^3,\ x(0) = x_*$.
%If the other row of $\nabla \frac{1}{2}d^2(\psi(x(t)), E)$ is $\Theta(t^d)$, 
%then we have
%\[
% \tp - p = \Theta(t^d).
%\]

\begin{example}
\label{ex:slowest_curve}
Let $E = \{U \in \S^3: \langle A_1, U\rangle = 1, \langle A_2, U\rangle = 0, \langle A_3,U\rangle = 0\}$, where
{\small \[
 A_1  = 
 \begin{pmatrix}
 1 & 0 & 0 \\
 0 & 0 & 1 \\
 0 & 1 & 0 
 \end{pmatrix},\
 A_2  = 
 \begin{pmatrix}
 0 & 0 & 1 \\
 0 & 1 & 0 \\
 1 & 0 & 0 
 \end{pmatrix},\ 
  A_3  = 
 \begin{pmatrix}
 0 & 0 & 0 \\
 0 & 0 & 0 \\
 0 & 0 & 1
 \end{pmatrix}.
\]}
Then $E$ is parameterized by
{\small \[
 \begin{pmatrix}
  1 & 0 & 0 \\
  0 & 0 & 0 \\
 0  & 0 & 0
 \end{pmatrix}
+ p_1
\begin{pmatrix}
 0 & 1 & 0\\
 1 & 0 & 0\\
 0 & 0 & 0
\end{pmatrix}
+ p_2
\begin{pmatrix}
 0 & 0 & -1\\
 0 & 2 & 0 \\
 -1 & 0 & 0
\end{pmatrix}
+ p_3
\begin{pmatrix}
 -2 & 0 & 0 \\
 0 & 0 & 1 \\
 0 & 1 & 0
\end{pmatrix}.
\]}
%$
%\begin{pmatrix}
%1 - 2 p_3 & p_1 & -p_2 \\
%p_1 & 2 p_2 & p_3 \\
%-p_2 & p_3 & 0 
%\end{pmatrix}.
%$
Now we have for $x_* = (1,0,0)$,
% Type 1, c1=0,c2=0,c3=1,c4=0,c5=1,c6=-1,c7=0,c8=1,m=1 ?
%{\small \begin{align*}
%& \nabla d^2(\psi(x), E) = \\
%& \frac{1}{3x_1^3}\begin{pmatrix} 
%-\left(2\,x_{13}^4+4\,x_{12}^2\,x_{13}^2-2\,x_{11}\,x_{12}\,x_{13}^2+x_{11}\,x_{13}^2+2\,x_{11}\,x_{12}^2\,x_{13}+2\,x_{12}^4-x_{11}\,x_{12}^2-2\,x_{11}^4+2\,x_{11}^3\right) \\
% 2\,x_{11}\,\left(2\,x_{12}\,x_{13}^2-x_{11}\,x_{13}^2+2\,x_{11}\,x_{12}\,x_{13}-x_{11}^2\,x_{13}+2\,x_{12}^3+2\,x_{11}^2\,x_{12}-x_{11}\,x_{12}\right) \\
% 2\,x_{11}\,\left(2\,x_{13}^3+2\,x_{12}^2\,x_{13}-2\,x_{11}\,x_{12}\,x_{13}+x_{11}\,x_{13}+x_{11}\,x_{12}^2-x_{11}^2\,x_{12}\right) 
%\end{pmatrix}.
%\end{align*}}
%
%
%%%%%%%%%%%%%%%%%%%%%%%%%%%%%%%%%%%%%%%%%%%%%%%%%%%%%%%%%%%% 
% Type 2, c1 = 0, c2 = 0, c3 = 1, c4 = 1, c5 = 0
{\small
\begin{align*}
& M(x_*)^{-1}\nabla_x \frac{1}{2}d^2(\psi(x), E)= \\
& \frac{1}{6x_{1}^3}\begin{pmatrix} 2\,x_{1}\,\left(2\,x_{2}\,x_{3}^2+2\,x_{1}\,x_{2}\,x_{3}+x_{1}^2\,x_{3}-x_{1}\,x_{3}+x_{2}^3\right) \\
 -2\,x_{1}\,\left(3\,x_{3}^3+2\,x_{2}^2\,x_{3}+2\,x_{1}^2\,x_{3}+x_{1}\,x_{2}^2+x_{1}^2\,x_{2}-x_{1}\,x_{2}\right) \\
 3\,x_{3}^4+4\,x_{2}^2\,x_{3}^2+2\,x_{1}\,x_{2}^2\,x_{3}-2\,x_{1}\,x_{2}\,x_{3}+x_{2}^4-x_{1}^4+x_{1}^3 
\end{pmatrix}.
\end{align*}
}%
%%%%%%%%%%%%%%%%%%%%%%%%%%%%%%%%%%%%%%%%%%%%%%%%%%%%%%%%%%%% 
%
%
We consider the system $M(x_*)^{-1}\nabla \frac{1}{2}d^2(\psi(x), E)=\vo$
around $x_* = (1,0,0)$.
By the rational transformation $(x_{1}, x_{2}, x_{3}) = (1/(1+z), u/(1+z), v/(1+z))$, 
we obtain
%
%%%%%%%%%%%%%%%%%%%%%%%%%%%%%%%%%%%%%%%%%%%%%%%%%%%%%%%%%%%%%
% Type 1, c1=0,c2=0,c3=1,c4=0,c5=1,c6=-1,c7=0,c8=1,m=1 ?
%\[
%\begin{cases}
% v\,w-u\,w+2\,v^3+2\,u\,v^2-v^2+2\,u^2\,v-v+2\,u^3+u^2+u =0\\
% v\,w+u\,w+2\,v^3-2\,u\,v^2+v^2+2\,u^2\,v-4\,u\,v+v-2\,u^3+u^2-3\,u =0\\
% v^2\,w-u^2\,w+2\,w+2\,v^4+4\,u^2\,v^2-2\,u\,v^2+2\,u^2\,v+2\,u^4-2=0
%\end{cases}
%\]
%
%%%%%%%%%%%%%%%%%%%%%%%%%%%%%%%%%%%%%%%%%%%%%%%%%%%%%%%%%%%%%
% Type 2, c1 = 0, c2 = 0, c3 = 1, c4 = 1, c5 = 0
\[
\begin{cases}
 f_1 := 2\,u\,v-v\,z+u^{3}+2\,u\,v^{2}  = 0,\\
 f_2 := -2\,v-u^{2}+u\,z-2\,u^{2}v-3\,v^{3} = 0,\\
 f_3 := z-2\,u\,v+2\,u^{2}v-2\,u\,v\,z+u^{4}+4\,u^{2}v^{2}+3\,v^{4} = 0
\end{cases}
\]
Since $x_*$ corresponds to $(u,v,z) = (0,0,0)$ and the Jacobian matrix
of $(f_1,f_2,f_3)^T$ at $(0,0,0)$ is 
$\left(
\begin{smallmatrix}
 0 & 0 & 0 \\
 0 & -2 & 0 \\
 0 & 0 & 1
\end{smallmatrix}\right)$, the equations $f_2 = f_3 = 0$ are solved
with respect to $v,z$ by convergent power series $v(u),z(u)$, respectively. 
%By substituting $(v,z)$ in $f_1$ with the power series, 
%we obtain $f_1 = O(u^d)$ for some $d>0$. 
%Instead of doing this, 
%we use \texttt{Singular} and calculate a weak normal form with the Mora's division algorithm and get
%\[
% q f_1 = a_2 f_2 + a_3 f_3 - \frac{9}{512}u^7 + \text{ higher order terms},
%\]
%where $q, a_2, a_3\in \R[u,v,z]$ and $q(0,0,0) \neq 0$.
%Here we use the negative degree reverse lexicographical ordering for $\R[z,v,u]$.
%This means that we can find convergent power series in $u$ that solve $(f_1,f_2,f_3) = (\Theta(u^7), 0, 0)$, and
%the degree $7$ is the highest such degree.
%
By applying the inverse of the rational transformation $(x_{1}, x_{2}, x_{3}) = (1/(1+z), u/(1+z), v/(1+z))$ to $v(u), z(u)$ and using $x_2=t$ as the new parameter, we obtain
\[
 x_1(t) = 1 + t^3 - 2t^6 - \frac{3t^7}{8} + O(t^8),\ 
x_2(t) = t,\ 
x_3(t) = -\frac{t^2}{2} + \frac{t^5}{2} + \frac{3t^6}{16} +  O(t^8).
\]
Then the curve $x(t) = (x_1(t),x_2(t),x_3(t))$ satisfies
\begin{equation}
 M(x(t))^{-1}\nabla_x \frac{1}{2}d^2(\psi(x(t)), E) = 
\begin{pmatrix}
 \frac{t^7}{8}\\
 0 \\
0
\end{pmatrix} + O(t^8).\label{eq:ex_slowest}
\end{equation}
Thus $x(t)$ is the minimizer of $\displaystyle \min_{\|x\|=\delta}\|M(x)^{-1}\nabla \frac{1}{2}d^2(\psi(x), E)\|$ for some $\delta>0$.
We can also determine the degree of the leading term of the first component of $(\ref{eq:ex_slowest})$ without actually calculating $x(t)$.
We use \texttt{Singular} and calculate a weak normal form with the Mora's division algorithm and get
\[
 q f_1 = a_2 f_2 + a_3 f_3 - \frac{9}{512}u^7 + \text{ higher order terms},
\]
where $q, a_2, a_3\in \R[u,v,z]$ and $q(0,0,0) \neq 0$.
Here we use the negative degree reverse lexicographical ordering for $\R[z,v,u]$.
This means that we can find convergent power series in $u$ that solve $(f_1,f_2,f_3) = (\Theta(u^7), 0, 0)$, and
the degree $7$ is the highest such degree.

With this $x(t)$, we define
\begin{align*}
& p(t) := \\
& \left(\frac{\left\langle B_i, \psi(x(t)) - U_*\right\rangle}{\|B_i\|^2}\right)_{i=1}^3
+ M(x(t))^{-1}\nabla_x \frac{1}{2}d^2(\psi(x(t)), E)
= \begin{pmatrix}
   t + \frac{t^7}{8}\\[1ex]
   \frac{t^2}{2} - \frac{t^5}{2} - \frac{t^6}{16}\\[1ex]
  -\frac{t^3}{2} + t^6 + \frac{3t^7}{16}
  \end{pmatrix} + O(t^8)
\end{align*}
Then we have
\[
 \phi(p(t)) = 
\begin{pmatrix}
 1  + t^3 - 2t^6 - \frac{3t^7}{8} & t + \frac{t^7}{8} 
& -\frac{t^2}{2} + \frac{t^5}{2} + \frac{t^6}{16} \\
\vphantom{\frac{\displaystyle A^2}{\displaystyle B}} t + \frac{t^7}{8} & t^2 - t^5 - \frac{t^6}{8} 
& -\frac{t^3}{2} + t^6 + \frac{3t^7}{16} \\
\vphantom{\frac{\displaystyle A^2}{\displaystyle B}}
-\frac{t^2}{2} + \frac{t^5}{2} + \frac{t^6}{16} 
& -\frac{t^3}{2} + t^6 + \frac{3t^7}{16} & 0
\end{pmatrix} + O(t^8).
\]
For sufficiently small $t>0$, since $\det\phi(p(t)) = \frac{t^{10}}{32} + O(t^{11}) > 0$ and the first eigenvalue of $\phi(p(t))$ is close to $1$,
we see that $P_{\S^3_+}\circ\phi(p(t))$ has rank $1$.
By reversing the modifications of the equations in the proof of Theorem $\ref{thm:analytic_det}$, we obtain the relation
%\[
% \psi(x(t)) = P_{\S^3_+}(\phi(p(t))).
%\]
$\displaystyle \phi(p(t)) \overset{P_{\S^3_+}}{\longmapsto} \psi(x(t))$. Then we have
\begin{align*}
\tp & := \phi^{-1}\circ P_E \circ P_{\S^n_+} \circ \phi(p(t)) \\
& = p(t) - M(x(t))^{-1}\nabla_x \frac{1}{2}d^2(\psi(x(t)), E)
 = p(t) + 
\begin{pmatrix}
 \frac{t^7}{8}\\
 0 \\
0
\end{pmatrix} + O(t^8).
\end{align*}
Moreover, since the leading terms of the second and the third coordinates of $p(t)$ have degree
$2$ and $3$ respectively, 
we have
\[
 \tp = p\left(t - \frac{t^7}{24}\right) + O(t^8),
\]
and thus
\[
  P_E\circ P_{\S^3_+}(\phi(p(t))) = \phi\left(p\left(t - \frac{t^7}{24}\right)\right) + O(t^8).
\]
Therefore, if we choose a matrix on the curve $\phi(p(t))$ that is sufficiently close to $U_*$, 
then the matrix mapped by $P_E\circ P_{\S^3_+}$ can be written as $\phi(p(t - t^7/24)) + O(t^8)$. This means that the one-step alternating projection moves the matrix in the slowest way in a sense.
In Section $\ref{sec:rate}$, we can actually prove that an AP sequence gives the slowest convergence rate if the initial point is taken from a neighborhood of the curve $\phi(p(t))$.
\end{example}

\section{Family of $3$-planes intersecting with $\S^3_+$ at a single point}
\label{sec:3-planes}
\subsection{Parametrization}
In Example $\ref{ex:slowest_curve}$, we constructed a candidate for a curve that gives the slowest convergence rate for given numeric matrices $A_1,A_2,A_3$.
To construct such a curve for a general case, we will obtain a parameterization of the family of $3$-planes that intersect with $\S^3_+$ at a single point. %The parametrization is used to give more detailed formula for an AP sequence.

\begin{prop}
\label{prop:3-planes}
Let $U_* = 
\left(
\begin{smallmatrix}
1 & 0 & 0 \\
0 & 0 & 0 \\
0 & 0 & 0  
\end{smallmatrix}
\right)
$.
A plane $E\subset \S^3$ satisfies $\S^3_+\cap E = \{U_*\}$ 
 and $\dim E = 3$ if and only if 
there exist $c_1,\ldots,c_8\in \R$ and an orthogonal matrix $\tP\in \R^{2\times 2}$ such that 
$E$ is written by
\[
  E = \{X \in \S^3:\langle A_1,X \rangle = 1,\ \langle A_2, X \rangle = 0,\ 
\langle A_3, X \rangle = 0\},
\]
where
$A_i$ are given as follows:
 \begin{align*}
 \intertext{Type $1$:}
& A_1 = 
P\begin{pmatrix}
 1 & c_1 & c_2 \\
 c_1 & c_3 & c_4 \\
 c_2 & c_4 & 0
 \end{pmatrix}P^T,\
 A_2 = 
P\begin{pmatrix}
 0 & c_5 & c_6 \\
 c_5 & c_7 & c_8 \\
 c_6 & c_8 & 0 
 \end{pmatrix}P^T,\
 A_3  = 
 P\begin{pmatrix}
 0 & 0 & 0 \\
 0 & \mu & 0 \\
 0 & 0 & 1
 \end{pmatrix}P^T,\\
& \mu >0,\ A_2\neq O,
 \intertext{or, Type $2$:}
& A_1  = 
 P\begin{pmatrix}
 1 & c_1 & c_2 \\
 c_1 & 0 & c_3 \\
 c_2 & c_3 & 0 
 \end{pmatrix}P^T,\
 A_2  = 
 P\begin{pmatrix}
 0 & 0 & c_4 \\
 0 & 1 & c_5 \\
 c_4 & c_5 & 0 
 \end{pmatrix}P^T,\ 
 A_3 = 
 P\begin{pmatrix}
 0 & 0 & 0 \\
 0 & 0 & 0 \\
 0 & 0 & 1
 \end{pmatrix}P^T,
 \end{align*}
where
$P = \text{{\scriptsize 
$
\left(\begin{array}{@{}c|c@{}}
1 & 
\begin{matrix}
0 & 0
\end{matrix}
\\
\hline
\begin{matrix}
0 \\ 0
\end{matrix} & 
%\raisebox{-0.5em}{\mbox{\huge $\tA$}}
\raisebox{-0.3em}{\mbox{\small $\tP$}}
\end{array}\right)
$}}$ in both types.
\end{prop}
\begin{proof}
If $E$ is a Type $1$ plane, then 
we can easily show $\S^3_+\cap E = \{U_*\}$
and $\dim E = 3$.
If $E$ is a Type $2$ plane,
then $X \in E$ is written as
\[
  X = P
\begin{pmatrix}
 1 - 2c_1s + 2c_2t - 2c_3u & s & t \\
 s & -2c_4t -2c_5u & u \\
 t & u & 0
\end{pmatrix}
P^T
\]
for $s,t,u\in \R$. Thus $\S^3_+\cap E = \{U_*\}$ and $\dim E = 3$.

Next, we will show the converse.
Suppose $\S^3_+ \cap E = \{U_*\}$ and $\dim E = 3$.
Then $E$ is contained in a supporting hyperplane to $\S^3_+$ at $U_*$.
 The set of normal vectors to a supporting hyperplane 
of $\S^3_+$ at $U_*$ is given by
\[
 N_{\S^3_+}(U_*) = \{A \in \S_+^3: \langle A, U_* \rangle = 0 \}
= \left\{
\left(\begin{array}{@{}c|c@{}}
0 & 
\begin{matrix}
0 & 0
\end{matrix}
\\
\hline
\begin{matrix}
0 \\ 0
\end{matrix} & 
%\raisebox{-0.5em}{\mbox{\huge $\tA$}}
\raisebox{-0.3em}{\mbox{\large $\tA$}}
\end{array}\right):\tA\in \S^2_+
\right\},
\]
see, e.g. \cite[Section 4.2.4]{HM}.
Then there exists $A_3 \in N_{\S^3_+}(U_*)$ such that $E \subset E_3$ for 
\[
 E_3 := \{X \in \S^3: \langle A_3, X - U_*\rangle = 0\}
= \{X \in \S^3: \langle A_3, X\rangle = 0\}.
\]
Here, for a matrix $X$, we partition $X$ as in the expression of $N_{\S^3_+}(U_*)$ and denote the right-lower part of $X$ as $\tX$.
%Let $A_3 \in N$ and define 
%\[
% E_3 := \{X \in \S^3: \langle A_3, X - U_*\rangle = 0\}
%= \{X \in \S^3: \langle A_3, X\rangle = 0\}.
%\]
%Then $E$ is the intersection of $E_3$ and two more hyperplanes. In the following, we consider two cases based on the rank of $\tA_3$.
%We note $A_3 \in N$ if and only if there exist
%$\lambda_1,\lambda_2\geq 0$ and an orthogonal matrix $\tP\in \R^{2\times 2}$ such that
%$P^TA_3P = 
%\left(
%\begin{smallmatrix}
% 0 & 0 & 0 \\
% 0 & \lambda_1 & 0 \\
% 0 & 0 & \lambda_2
%\end{smallmatrix}
%\right)
%$
%for 
%$P = \text{{\scriptsize 
%$
%\left(\begin{array}{@{}c|c@{}}
%1 & 
%\begin{matrix}
%0 & 0
%\end{matrix}
%\\
%\hline
%\begin{matrix}
%0 \\ 0
%\end{matrix} & 
%%\raisebox{-0.5em}{\mbox{\huge $\tA$}}
%\raisebox{-0.3em}{\mbox{\small $\tP$}}
%\end{array}\right)
%$}}$.
%
%

If $\rank \tA_3 = 2$, then there exists an orthogonal matrix $\tP\in \R^{2\times 2}$
such that 
$\Lambda_3 :=
\left(
\begin{smallmatrix}
 0 & 0 & 0\\
 0 & \lambda_1 & 0\\
 0 & 0 & \lambda_2
\end{smallmatrix}
\right) =
P^TA_3P 
$ for 
$P = \text{{\scriptsize 
$
\left(\begin{array}{@{}c|c@{}}
1 & 
\begin{matrix}
0 & 0
\end{matrix}
\\
\hline
\begin{matrix}
0 \\ 0
\end{matrix} & 
%\raisebox{-0.5em}{\mbox{\huge $\tA$}}
\raisebox{-0.3em}{\mbox{\small $\tP$}}
\end{array}\right)
$}}$.
We may assume $\lambda_2 = 1$.
Then we have
\begin{align*}
 X \in \S^3_{+} \cap E_3 & \Longleftrightarrow X\in \S^3_+,\ \langle \tA_3, \tX \rangle = 0
% \Longleftrightarrow X \in \S^3_+,\ \tA_1\tX = \tX\tA_1 = O,\\
 \Longleftrightarrow X = 
\left(\begin{array}{@{}c@{}|c@{}}
x_{11} & 
\begin{matrix}
0 & 0
\end{matrix}
\\
\hline
\begin{matrix}
0 \\ 0
\end{matrix} & 
\raisebox{-0.3em}{\mbox{\large $O$}}
\end{array}\right),\ x_{11}\geq 0.
\end{align*}
Thus $\S^3_+ \cap E_3$ is a half line. 
Since $\S^3_+ \cap E = \{U_*\}$ and $\dim E = 3$, there exist hyperplanes $E_1 = \{X\in \S^3:\langle A_1, X\rangle = 1\}$ and $E_2 = \{A_2\}^\perp$ such that $E = E_1\cap E_2 \cap E_3$.
Since $P U_* P^T = U_* \in E$, we see that $A_1, A_2, A_3$ 
satisfy
%$E = \{X \in \S^3:\langle A_1,X \rangle = 1,\ \langle A_2, X \rangle = 0,\ 
%\langle A_3, X \rangle = 0\}$, 
%where
 \begin{align*}
 & 
A_1  = 
P^T\begin{pmatrix}
 1 & c_1 & c_2 \\
 c_1 & c_3 & c_4 \\
 c_2 & c_4 & c_5 
\end{pmatrix}P,\
A_2  = 
P^T\begin{pmatrix}
 0 & c_6 & c_7 \\
 c_6 & c_8 & c_9 \\
 c_7 & c_9 & c_{10} 
\end{pmatrix}P,\ 
A_3  = 
P^T\begin{pmatrix}
 0 & 0 & 0 \\
 0 & \mu & 0 \\
 0 & 0 & 1
\end{pmatrix}P,
\\
& 
\mu > 0,\ 
\{A_1,A_2,A_3\}: \text{linearly independent},
\end{align*}
for some $\mu, c_i \in \R$.
By deleting redundant parameters, we obtain the matrices in Type $1$.

%%%%%%%%%%%%%%%%%%%%%%%%%%%%%%%%%%%%%%%%%%%%%%%%%%%%%%%%%%%%

If $\rank \tA_3 = 1$, then
there exist
$\lambda > 0$ and an orthogonal matrix $\tP\in \R^{2\times 2}$ such that
$A_3 = 
P\left(
\begin{smallmatrix}
 0 & 0 & 0 \\
 0 & 0 & 0 \\
 0 & 0 & \lambda
\end{smallmatrix}
\right)P^T
$
for 
$P = \text{{\scriptsize 
$
\left(\begin{array}{@{}c|c@{}}
1 & 
\begin{matrix}
0 & 0
\end{matrix}
\\
\hline
\begin{matrix}
0 \\ 0
\end{matrix} & 
%\raisebox{-0.5em}{\mbox{\huge $\tA$}}
\raisebox{-0.3em}{\mbox{\small $\tP$}}
\end{array}\right)
$}}$.
We may assume $\lambda = 1$.
%Since $P^T U_* P = U_*\in E$ and $A_3 = 
%P\left(
%\begin{smallmatrix}
% 0 & 0 & 0 \\
% 0 & 0 & 0 \\
% 0 & 0 & 1
%\end{smallmatrix}
%\right)P^T
%$, 
By a similar argument above, we see that $A_1, A_2, A_3$ are given by
\[
A_1 = P\begin{pmatrix}
 1 & c_1 & c_2 \\
 c_1 & d_1 & c_3 \\
 c_2 & c_3 & 0 
\end{pmatrix}P^T,\ 
A_2 = 
P\begin{pmatrix}
 0 & d_2 & c_4 \\
 d_2 & d_3 & c_5 \\
 c_4 & c_5 & 0 
\end{pmatrix}P^T,\ 
A_3 = 
P\begin{pmatrix}
 0 & 0 & 0 \\
 0 & 0 & 0 \\
 0 & 0 & 1
\end{pmatrix}P^T
\]
for some $c_i, d_j\in \R$. In addition, for
$X = 
P\left(
\begin{smallmatrix}
 x_{11} & x_{21} & x_{31} \\
 x_{21} & x_{22} & x_{23} \\
 x_{31} & x_{23} & x_{33}
\end{smallmatrix}
\right)P^T\in E,
$
we have
\[
 \begin{cases}
 x_{11} + 2c_1 x_{21} + 2c_2x_{31} + d_1 x_{22} + 2c_3x_{23} = 1,\\
 2 d_2 x_{21} + 2c_4 x_{31} + d_3 x_{22}  + 2c_5 x_{23} = 0,\quad x_{33} = 0.
\end{cases}.
\]
If $d_3 = 0$, then 
$P\left(\begin{smallmatrix}
  1 - d_1t & 0 & 0 \\
  0 & t & 0 \\
  0 & 0 & 0
 \end{smallmatrix}\right)P^T \in \S^3_+\cap E$ for $t\in \R$
and hence it contradicts to $\S^3_+\cap E = \{U_*\}$.
Thus $d_3 \neq 0$. In addition, if $d_2 \neq 0$, then 
$P\left(
\begin{smallmatrix}
  1 - d_1t + (c_1d_3/d_2) t& -(d_3/2d_2)t & 0 \\
  -(d_3/2d_2)t & t & 0 \\
  0 & 0 & 0
 \end{smallmatrix}
\right)P^T
 \in E$ for $t\in \R$.
%$\begin{pmatrix}
%  1 - d_1t + (c_1d_3/d_2) t& -(d_3/2d_2)t & 0 \\
%  -(d_3/2d_2)t & t & 0 \\
%  0 & 0 & 0
% \end{pmatrix} \in E$ for $t\in \R$.
Since this matrix is positive semidefinite for sufficiently small $t>0$,
it is a contradiction and hence $d_2 = 0$.
Therefore we can write
$A_2 = 
P\left(
\begin{smallmatrix}
  0 & 0 & c_4 \\
 0 & 1 & c_5 \\
 c_4 & c_5 & 0 
\end{smallmatrix}
\right)P^T
$ and hence 
$A_1 = 
P\left(
\begin{smallmatrix}
  1 & c_1 & c_2 \\
 c_1 & 0 & c_3 \\
 c_2 & c_3 & 0 
\end{smallmatrix}
\right)P^T
$ by deleting $(2,2)$ element of $A_1$ using $A_2$.

%%%%%%%%%%%%%%%%%%%%%%%%%%%%%%%%%%%%%%%%%%%%%%%%%%%%%%%%%%%%

\end{proof}

\begin{remark}
[relations to singularity degrees]
\label{rem:sd}
 It is well known that an upper bound of the convergence rate of alternating projections of 
 \[
 E=\{X\in \S^3:\langle A_1,X \rangle = 1,\ \langle A_2, X \rangle = 0,\ 
 \langle A_3, X \rangle = 0\}
 \]
 and $\S^3_+$ is given using the singularity degree; \cite{DLW2017}. %The \textit{singularity degree} is a nonnegative integer determined by the iterative process called \textit{facial reduction}. For the detail; see, e.g. \cite{BW, DW, P}. 
 %We note that the singularity degree of the intersection of an affine subspace  and $\S^n_+$ is less than or equal to $n-1$. 
 As explained in Example $\ref{ex:linear}$, 
 the singularity degree of $E\cap \S^3_+$ is either $0, 1$ or $2$.
 A Type $1$ plane in Proposition $\ref{prop:3-planes}$ has the singularity degree $1$ since $A_3$ itself is positive semidefinite.
 For a Type $2$ plane, a linear combination of $A_2$ and $A_3$ can be positive semidefinite if and only if $c_4 = 0$, in case the singularity degree is $1$.
 Therefore the singularity degree is $2$ if and only if $c_4 \neq 0$.
\end{remark}

\subsection{Pl\"ucker embedding}
\label{sec:plucker}
Let $\cD$ be the family of $3$-planes that intersect with $\S^3_+$ at 
$U_* = 
\left(
\begin{smallmatrix}
1 & 0 & 0 \\
0 & 0 & 0 \\
0 & 0 & 0 
\end{smallmatrix}
\right)$. Let $B_1,B_2,B_3$ be a basis of the linear space
\[
E'=\{X\in \S^3: \langle A_1,X \rangle = 0,\ \langle A_2, X \rangle = 0,\ 
\langle A_3, X \rangle = 0\}.
\]
For the standard basis $e_1,\ldots,e_6\in \S^3$,
the Pl\"ucker coordinate $\left\{C_{i_1,i_2,i_3}\right\}$ is given by
\[
 B_1 \wedge B_2 \wedge B_3
= \sum\{C_{i_1,i_2,i_3} e_{i_1}\wedge e_{i_2}\wedge e_{i_3}:1\leq i_1<i_2<i_3\leq 6\} \in \bigwedge^3\S^3.
\]
With these coordinates, $\cD$ is considered as a subset of the Grassmannian of $3$-planes in $\S^3$. 
By counting the number of parameters and noting that $P$ is a one-parameter matrix,
we see that the dimensions of the family of planes of Type $2$ is $6$.
Consider the dimensions of the family of planes of Type $1$. Since $A_2 \neq O$, one of $c_5,c_6,c_7,c_8$ can be replaced with $1$ and then the corresponding parameter in $A_1$ is redundant. Thus we conclude that the dimension of the family of planes of Type $1$ is $8$. 

We can explicitly calculate the defining ideals of these family of planes with \texttt{Macaulay 2} as follows.
First, define the ring with \texttt{SkewCommutative} elements $e_1,\ldots,e_6$.
We write $B_1,B_2,B_3$ as linear combinations of $e_1,\ldots,e_6$, and then
the coefficients of the product of these linear combinations are the coordinates $C_{i_1,i_2,i_3}$.
Let
\begin{align*}
  I & = \left\langle t\cdot C_{i_1,i_2,i_3} - x_{i_1,i_2,i_3} : 1\leq i_1<i_2<i_3 \leq 6 \right\rangle + \langle u_1^Tu_2,\|u_1\|^2 - 1, \|u_2\|^2 - 1 \rangle\\
& \subset \R[c_1,\ldots,c_{5},t,x_{1,2,3},\ldots,x_{4,5,6}].
\end{align*}
Then the ideal
\[
 J = I\cap \R[x_{1,2,3},\ldots, x_{4,5,6}]
\]
is the defining ideal of the Pl\"ucker embedding of the $3$-planes of Type $2$.
We can calculate $J$ with the command \texttt{elimination} in \texttt{Macaulay 2}.
Similarly, we obtain the defining ideal of the $3$-planes of Type $1$.

Then we have the following relation:
\[
 \begin{array}{lllllll}
\text{Gr}(3,6)
&
\supset
&
\cD
&
=
&
  \{
\text{Type $1$}
\}
&
\cup
&
\{
\text{Type $2$}
\}\\
\dim 9
&
&
&
&
\text{SD}=1,\ \dim 8, \text{ semialg.}
& 
&
\text{SD}=2,\ \dim 6\\
& & & & & &
\text{SD}=1,\ \dim 5
 \end{array}
\]
The Grassmannian of $3$-planes in $\S^3$ has dimension $9$.
$\cD$ is a semialgebraic subset with dimension $8$,
in which a generic plane intersects $\S^3_+$ with singularity degree $1$.
The planes with singularity degree $2$ form a $6$ dimensional subvariety.
Within the subvariety, there is a $5$ dimensional subvariety
whose point corresponds to a plane with singularity degree $1$.

%\subsection{Connections to sigularity types}

\section{Rational formulas for a special curve}
\label{sec:rat}

%In this section, we present a curve giving the slowest convergence rate for a Type 2 plane with $c_4\neq 0$. 
In this section, we consider a Type $2$ plane with $c_4\neq 0$. As explained in Remark $\ref{rem:sd}$, the set of such planes is exactly the set of 3-planes whose intersections with $\S^3_+$ are $\{U_*\}$ and have singularity degree $2$.

A matrix in a Type $2$ plane is written as
% \begin{pmatrix} 
%1 -2\,c_{1}\,p_{1} +2\,c_{2}\,p_{2} -2\,c_{3}\,p_{3}
%& p_{1} & -p_{2}\\ p_{1} 
%& 2\,c_{4}\,p_{2}-2\,c_{5}\,p_{3} 
%& p_{3}\\ -p_{2} & p_{3} & 0 \end{pmatrix}.
\[
U_* + p_1 B_1 + p_2B_2 + p_3 B_3,
\]
where $U_* = 
\left(
\begin{smallmatrix}
 1 & 0 & 0 \\
0 & 0 & 0\\
0 & 0 & 0
\end{smallmatrix}\right)$ and
\begin{equation}
 B_1 = \begin{pmatrix} -2c_{1} & 1 & 0\\ 1 & 0 & 0\\ 0 & 0 & 0 \end{pmatrix},\quad
B_2 = \begin{pmatrix} 2c_{2} & 0 & -1\\ 0 & 2c_{4} & 0\\ -1 & 0 & 0 \end{pmatrix},\quad
B_3 = \begin{pmatrix} -2c_{3} & 0 & 0\\ 0 & -2c_{5} & 1\\ 0 & 1 & 0 \end{pmatrix},\label{basis}
\end{equation}
and $c_4 \neq 0$.
Note that $B_1, B_2, B_3$ are not necessarily orthogonal to each other.

%\subsection{Rational formulas}
\subsection{Special curve}
First, we present the special curve $G(t)$ for a Type $2$ plane with $c_4\neq 0$, 
which is, in fact, the slowest curve for such a plane, as obtained in Example $\ref{ex:slowest_curve}$ 
and will be used to show the tightness of a known upper bound for the convergence rate of alternating projections in Section $\ref{sec:rate}$. 
Then for a matrix $G(t)$ on the curve, we give rational formulas for 
$P_{\S^3_+}(G(t))$ and 
$P_E\circ P_{\S^3_+}(G(t))$ up to degree $7$.\\

%\noindent\textbf{The special curve $G(t)$.}
%Then a curve giving the slowest convergence rate is given as follows.
Let 
\begin{align}
g_{13} = \tg_{13} + r_{13},\quad g_{23} = -\frac{\tg_{13}}{w}t + r_{23},
%\begin{split}
%   g_{13} = & \frac{1}{2 c_4 w + 2 c_5 t}\,t^2
% - \frac{2(2c_5^2 + 1)}{(2 c_4 w + 2 c_5 t)^5}\,t^6 + r_{13},\\ 
% %
% g_{23} = & -\frac{1}{(2 c_4 w +  2c_5 t)w}\,t^3
% +\frac{2\left(2\,c_{5}^2+1\right)}{(2c_4w + 2c_5 t)^5w}\,t^7 + r_{23},
%\end{split}
\label{eq:slowest_g}
\end{align}
where
\begin{align} 
\tg_{13} & = \frac{1}{2 c_4 w + 2 c_5 t}\,t^2 - \frac{2(2c_5^2 + 1)}{(2 c_4 w + 2 c_5 t)^5}\,t^6,\notag \\
  r_{13} & =   
 \frac{c_5}{c_4}r_0 t^7 %\left(\frac{c_4\langle B_3, B_1\rangle + c_5\langle B_2, B_1\rangle}{8c_4^5\|B_1\|^2} + \frac{1}{8c_4^3}\right)t^7
+ \frac{\langle B_2,B_1\rangle}{16\,c_4^6\,\|B_1\|^2}t^7,\quad
 r_{23} =  
- \frac{2c_5}{(2c_4 w + 2c_5 t)^4 w}\,t^6 + r_0 t^7,
 %\left(\frac{c_4\langle B_3,B_1\rangle + c_5\langle B_2,B_1\rangle}{8c_4^5\|B_1\|^2} + \frac{1}{8c_4^3}\right)t^7. 
 \notag\\
 r_0 & = \left(\frac{c_4\langle B_3, B_1\rangle + c_5\langle B_2, B_1\rangle}{8c_4^5\|B_1\|^2} + \frac{1}{8c_4^3}\right), \notag\\
 w & = 1 -2\,c_{1}\,t
+\frac{c_{2}\,t^2}{c_{4}\,\left(1 -2\,c_{1}\,t + 
\frac{c_{2}\,t^2}{c_{4}\,\left(1-2\,c_{1}\,t\right)+c_{5}\,t} +\frac{c_{3}\,t^3}{c_{4}}\right)+c_{5}\,t} \label{eq:w}
\\
 & \hspace{3em}+\frac{c_{3}\,t^3}{\left(c_{4}\,\left(1 - 2\,c_1\,t + \frac{c_{2}\,t^2}{c_{4}}\right)+c_{5}\,t\right)\,
\left(
1 - 2\,c_1\,t + \frac{c_{2}\,t^2}{c_{4}}\right)},\notag 
\end{align}
%%%%%%%%%%%%%%%%%%%%%%%%%%%%%%%%%%%%%%%%%%%%%%%%%%%%%%%%%%%%
%
%\begin{align*} 
% w = &
%1 -2\,c_{1}\,t
%%
%+{{c_{2}\,t^2}\over{c_{4}\,\left(1 -2\,c_{1}\,t + 
%{{c_{2}\,t^2}\over{c_{4}\,\left(1-2\,c_{1}\,t\right)+c_{5}\,t}} +{{c_{3}\,t^3}\over{c_{4}}}\right)+c_{5}\,t}}
%\\
%%
%& +{{c_{3}\,t^3}\over{\left(c_{4}\,\left(1 - 2\,c_1\,t + {{c_{2}\,t^2}\over{c_{4}}}\right)+c_{5}\,t\right)\,
%%
%\left(
%%{{c_{3}\,t^3}\over{\left(c_{4}\,\left(1-2\,c_{1}\,t\right)+c_{5}\,t\right)\,\left({{c_{2}\,t^2}\over{c_{4}\,\left(1-2\,c_{1}\,t\right)+c_{5}\,t}}+{{c_{3}\,t^3}\over{c_{4}}}-2\,c_{1}\,t+1\right)}}+
%%
%1 - 2\,c_1\,t + {{c_{2}\,t^2}\over{c_{4}
%%\,\left({{c_{2}\,t^2}\over{c_{4}}}-2\,c_{1}\,t+1\right)+c_{5}\,t
%}}\right)}}\\
%%
%  \tg_{13} = & \frac{t^2}{2 c_4 w + 2 c_5 t}
%- \frac{2(2c_5^2 + 1)t^6}{(2 c_4 w + 2 c_5 t)^5},\\ 
%%
% \tg_{23} = & -\frac{t^3}{(2 c_4 w +  2c_5 t)w}
% +{{2\left(2\,c_{5}^2+1\right)\,t^7}\over{(2c_4w + 2c_5 t)^5w}},\\
%  g_{13} = &  \tg_{13}
%+ \frac{c_5}{c_4}\left(\frac{c_4\langle B_3, B_1\rangle + c_5\langle B_2, B_1\rangle}{8c_4^5\|B_1\|^2} + \frac{1}{8c_4^3}\right)t^7
%+ \frac{\langle B_2,B_1\rangle}{16\,c_4^6\,\|B_1\|^2}t^7,\\
%%
% g_{23} = & \tg_{23}
%+ \left(\frac{c_4\langle B_3,B_1\rangle + c_5\langle B_2,B_1\rangle}{8c_4^5\|B_1\|^2} + \frac{1}{8c_4^3}\right)t^7 
% - \frac{2c_5t^6}{(2c_4 w + 2c_5 t)^4 w},\\
%g(t) = & (t, g_{13}, g_{23}).
%\end{align*}
%
%%%%%%%%%%%%%%%%%%%%%%%%%%%%%%%%%%%%%%%%%%%%%%%%%%%%%%%%%%%%
%
Define $g(t) = (t, g_{13}, g_{23})$, $\displaystyle h = \frac{2t^6}{(2c_4 w + 2c_5 t)^4w}
$ and
\begin{equation}
  G(t) := \phi(g(t)) = 
\begin{pmatrix} 
1 - 2c_1 t + 2c_2 g_{13} - 2c_3 g_{23} & t & -g_{13}\\ 
t & 2c_4 g_{13}-2c_5 g_{23} & g_{23}\\ 
-g_{13} & g_{23} & 0 
\end{pmatrix}.\label{eq:slowest}
\end{equation}
\begin{example}
\label{ex:moment}
 For $c_1 = c_4 = 1,\ c_2 = c_3 = c_5 = 0$, we have $w=1 - 2t$ and
 \begin{align*}
 g_{13} = & {{t^2}\over{2\,\left(1-2\,t\right)}} + {{t^6}\over{16\,\left(1-2\,t\right)^5}},\quad
 g_{23} =  -{{t^3}\over{2\,\left(1-2\,t\right)
 ^2}} + {{3\,t^7}\over{16\,\left(1-2\,t\right)^6}},\\
 G(t) = & 
 \begin{pmatrix} 
 1 - 2 t 
 & * & *\\ 
 t & 2 g_{13} & *\\ 
 -g_{13} & g_{23} & 0 
 \end{pmatrix},
 \end{align*} 
 where $*$ in the matrix stands for the corresponding element of the transpose of the lower triangular block part.
By changing the variable as $s = t/(1 - t)$, we see that
\[
 \frac{G(t)}{1 - 2t}
=
\begin{pmatrix}
 1 & * & * \\
 s & s^2  & *\\
-\frac{1}{2}s^2 & -\frac{1}{2}s^3 & 0
\end{pmatrix} + 
\begin{pmatrix}
 0 & * & * \\
 0 & - \frac{1}{8}s^6 & *\\
 \frac{1}{16}s^6 & \frac{3}{16}s^7 & 0
\end{pmatrix} + O(s^8).
\]
This is the curve obtained by perturbing a moment curve with the higher order terms.
\end{example}

%\noindent\textbf{Rational formulas.}
\subsection{Rational formulas}
The following two formulas give rational expressions for projections $P_{\S^3_+}$ and
$P_E\circ P_{\S^3_+}$ along the curve $G(t)$ up to degree $7$. 
The proofs are given in the next section.
\begin{thm}
\label{thm:ratPS}
For sufficiently small $t>0$, we have
\begin{align}
 &  P_{\S^3_+}(G(t)) = G(t) +
\begin{pmatrix}
 0 & * & * \\
\displaystyle -\frac{t^7}{8c_4^4} 
& \displaystyle h - \frac{\langle B_2,B_1\rangle}{8\,c_4^5\,\|B_1\|^2}t^7 
& *
\\[2ex]
\displaystyle c_4 h 
& \displaystyle 
 c_5 h - \frac{c_5\langle B_2,B_1\rangle}{8\,c_4^5\,\|B_1\|^2}t^7- \frac{\langle B_3,B_1\rangle}{8c_4^4\|B_1\|^2}t^7  
& \displaystyle  \frac{g_{13}^2}{w}
\end{pmatrix} + O(t^8).\label{eq:ratPS}
\end{align}
\end{thm}

\begin{thm}
\label{thm:ratPEPS}
For sufficiently small $t>0$, we have
\begin{align}
    P_E\circ P_{\S^3_+}(G(t)) %& = U(t) -\frac{t^7}{4c_4^4\|B_1\|^2}B_1 + O(t^8)\\
& = G\left(t - \frac{t^7}{4c_4^4\|B_1\|^2}\right) +  O(t^8).\label{eq:ratPEPS}
\end{align}

\end{thm}

%\begin{remark}
%In the analysis of an AP sequence, 
%it has a crucial role to estimate  $d(U, \S^3_+)$ and $d(P_{\S^3_+}(U), E)$.
%When $U = G(t)$, we can explicitly calculate them applying Theorem $\ref{thm:ratPS}$ and $\ref{thm:ratPEPS}$ to
%\begin{align*}
%  d(G(t), \S^3_+) & = \|G(t) - P_{\S^3_+}(G(t))\|,\\
%d(P_{\S^3_+}(G(t)), E) & = \|P_{\S^3_+}(G(t)) - P_E\circ P_{\S^3_+}(G(t))\|.
%\end{align*}
%\end{remark}

\subsection{Proof of the formulas}
\label{sec:proof}
We use four lemmas to prove the formula.
Recall that the rational function $w$ in $(\ref{eq:w})$ is used to define $g_{13}$ and $g_{23}$ as in $(\ref{eq:slowest_g})$.
We start with the lemma that shows that $w$ has the following recursive property.
\begin{lemma}
Let $w$ is the rational function in $(\ref{eq:w})$. Then
\label{lemma:w}
 \[
  w = 1 - 2c_1 t + 2c_2 g_{13} - 2c_3 g_{23} + O(t^6).
 \]
\end{lemma}
\begin{proof}
 Let $w_i$ be the sum of the terms of $w$ with the degree less than or equal to $i$. Then 
\begin{align*}
  w_2 & = 1 - 2 c_1 t + \frac{c_2}{c_4}t^2,\\
 w_3 & = 1 - 2 c_1 t + \frac{c_2}{c_4(1 - 2c_1 t) + c_5 t}t^2 + \frac{c_3}{c_4}t^3.
\end{align*}
By considering the Taylor expansion, we see that
\begin{align*}
 1 - 2c_1 t + 2c_2 g_{13} - 2c_3 g_{23}
& = 1 - 2c_1 t + \frac{2 c_2}{2 c_4 w_3 + 2 c_5 t}\,t^2
  -\frac{2c_3}{(2 c_4 w_2 +  2c_5 t)w_2}\,t^3 + O(t^6)\\
& = w + O(t^6).
\end{align*}
\end{proof}

Here, we write $g_{13} = g_{13}(w)$ and $g_{23} = g_{23}(w)$ to specify $w$ in their definitions.
Let 
\[
\hw := 1 - 2c_1 t + 2c_2 g_{13}(w) - 2c_3 g_{23}(w).
\] 
By Lemma $\ref{lemma:w}$, $\hw$ equals to $w$ up to degree $5$.
By the similar argument to the proof of Lemma $\ref{lemma:w}$, we have $g_{13}(\hw) = g_{13}(w) + O(t^8)$, $g_{23}(\hw) = g_{23}(w) + O(t^9)$ and hence
\[
 \hw = 1 - 2c_1 t + 2c_2 g_{13}(\hw) - 2c_3 g_{23}(\hw) + O(t^8).
\]
%the replacement of $w$ with $w'$ does not change $1 - 2c_1 t + 2c_2 g_{13} - 2c_3 g_{23}$ up to degree $7$.
%Note that 
%In the following, we denote the truncation of $w'$ to degree $7$ by $w$.
Thus, if $\hw$ is used in the definition $(5)$ of $G(t)$ instead of $w$, then $G(t)$ does not change up to degree $7$.
Since the following arguments only treat equations up to degree $7$ except for Lemma $\ref{lemma:det}$, we can replace $w$ with $\hw$. To be precise, we keep using $w$ and $\hw$ separately, but the reader may consider $w$ as $\hw$.\\

\noindent\textbf{Decomposition.}
We will investigate the basic structure of the matrix $G(t)$, such as a decomposition of the low degree terms, the first eigenvalue, and the determinant.
First, we will show that $G(t)$ is obtained by perturbing a rank $1$ matrix
with higher order terms. Recall that $\displaystyle h = {{2t^6}\over{(2c_4 w + 2c_5 t)^4w}}$.
\begin{lemma}
\label{lemma:rank1}
\[
 G(t) = 
\hw\begin{pmatrix}
 1 \\
 \frac{t}{w} \\
 \frac{-g_{13}}{w}
\end{pmatrix}
\begin{pmatrix}
 1 &  \frac{t}{w} &
 \frac{-g_{13}}{w}
\end{pmatrix} + 
\begin{pmatrix}  
  0 & 0 & 0 \\
  0 & \displaystyle -h + \frac{\langle B_2,B_1\rangle}{8\,c_4^5\,\|B_1\|^2}t^7 & \displaystyle g_{23} + \frac{tg_{13}}{w} \\[2ex]
  0 & \displaystyle g_{23} + \frac{tg_{13}}{w} & \displaystyle-\frac{g_{13}^2}{w} 
 \end{pmatrix} + O(t^8).
\]
In addition, 
\[
 \begin{pmatrix}  
  0 & 0 & 0 \\
  0 & -h + \frac{\langle B_2,B_1\rangle}{8\,c_4^5\,\|B_1\|^2}t^7 &  g_{23} + \frac{tg_{13}}{w} \\
  0 & g_{23} + \frac{tg_{13}}{w} & -\frac{g_{13}^2}{w} 
 \end{pmatrix}
= \begin{pmatrix}
 0 & 0 & 0 \\
 0 & O(t^6) & O(t^6) \\
 0 & O(t^6) & O(t^4) 
\end{pmatrix}.
\]
\end{lemma}
\begin{proof}
Recall $\displaystyle \tg_{13} =  \frac{1}{2 c_4 w + 2 c_5 t}\,t^2
- \frac{2(2c_5^2 + 1)}{(2 c_4 w + 2 c_5 t)^5}\,t^6$. Let $\tg_{23} = -\frac{t\tg_{13}}{w}$.
%\begin{align*}
%   \tg_{13} := & \frac{1}{2 c_4 w + 2 c_5 t}\,t^2
%- \frac{2(2c_5^2 + 1)}{(2 c_4 w + 2 c_5 t)^5}\,t^6,\\ 
%%
% \tg_{23} := & -\frac{1}{(2 c_4 w +  2c_5 t)w}\,t^3
% +{{2\left(2\,c_{5}^2+1\right)}\over{(2c_4w + 2c_5 t)^5w}}\,t^7 = -\frac{t %\tg_{13}}{w}.
%\end{align*}
Then we have
\begin{align*}
  2c_4\tg_{13}-2c_5\tg_{23} = 2c_4\tg_{13}+2c_5\frac{t\tg_{13}}{w}
= \frac{2c_4 w + 2c_5 t}{w}\tg_{13}
& = \frac{t^2}{w} 
- \frac{2(2c_5^2 + 1)t^6}{(2 c_4 w + 2 c_5 t)^4w} \\
& = \frac{t^2}{w} - (2c_5^2 + 1)h.
\end{align*}
In addition, 
\[
2c_4 r_{13} - 2c_5 r_{23}
= \frac{2c_4\langle B_2,B_1\rangle}{16\,c_4^6\,\|B_1\|^2}t^7
+ \frac{4c_5^2t^6}{(2c_4 w + 2c_5 t)^4 w} 
= \frac{\langle B_2,B_1\rangle}{8\,c_4^5\,\|B_1\|^2}t^7
 + 2c_5^2h.
\]
Since $g_{23} = \tg_{23} + r_{23}$, we obtain
\begin{equation}
 2c_4g_{13}-2c_5g_{23} 
 % & \frac{t^2}{w} - {{2\left(2\,c_5^2+1\right)t^6}\over{(2c_4w + 2c_5 t)^4w}}
% + \frac{2c_4\langle B_2,B_1\rangle}{16\,c_4^6\,\|B_1\|^2}t^7 + \frac{4c_5^2t^6}{(2c_4 w + 2c_5 t)^4 w}  \\
%
% = & \frac{t^2}{w} - {{2t^6}\over{(2c_4w + 2c_5 t)^4w}} + \frac{\langle B_2,B_1\rangle}{8\,c_4^5\,\|B_1\|^2}t^7 \notag \\
%
= 
\frac{t^2}{w} 
+ \frac{\langle B_2,B_1\rangle}{8\,c_4^5\,\|B_1\|^2}t^7 - h.\label{eq:U22} 
\end{equation}
Since 
$\hw \frac{t^2}{w^2} = \frac{t^2}{w} + O(t^8)$, 
we have
\[
 G(t) -
\hw\begin{pmatrix}
 1 \\
 \frac{t}{w} \\[1ex]
 \frac{-g_{13}}{w}
\end{pmatrix}
\begin{pmatrix}
 1 &  \frac{t}{w} &
 \frac{-g_{13}}{w}
\end{pmatrix}
= \begin{pmatrix}  
  0 & 0 & 0 \\
  0 
  & -h
+ \frac{\langle B_2,B_1\rangle}{8\,c_4^5\,\|B_1\|^2}t^7
  & \displaystyle g_{23} + \frac{tg_{13}}{w} \\[2ex]
  0 
  & \displaystyle g_{23} + \frac{tg_{13}}{w} & \displaystyle -\frac{g_{13}^2}{w} 
 \end{pmatrix} + O(t^8).
\] 
In addition, 
\begin{align}
 g_{23} & = -\frac{tg_{13}}{w}
+ \left(\frac{c_4\langle B_3,B_1\rangle + c_5\langle B_2,B_1\rangle}{8c_4^5\|B_1\|^2} + \frac{1}{8c_4^3}\right)t^7 
 - c_5 h \label{eq:G23}\\
& = -\frac{tg_{13}}{w}
+ \frac{c_4\langle B_3,B_1\rangle + c_5\langle B_2,B_1\rangle}{8c_4^5\|B_1\|^2}t^7 
 + (-c_5  + c_4 t)h + O(t^8). \notag
\end{align}
In the last equality, we use $\frac{1}{8c_4^3}t^7 = c_4 t h = O(t^8)$.
Thus we have 
\begin{equation}
 g_{23} + \frac{t g_{13}}{w} = 
 -c_5h + O(t^7) = O(t^6).\label{eq:G2313}
\end{equation}
\end{proof}

\noindent\textbf{The first eigenvalue and eigenvector.} Using the decomposition of the low degree term, we obtain the first eigenvalue and the associated eigenvector of $G(t)$ up to degree $7$.
\begin{lemma}
\label{lemma:approxeig}
Let 
\[
 v =  
\begin{pmatrix}
 1\\
\displaystyle \frac{t}{\hw}\\[2ex]
\displaystyle \frac{-g_{13}}{w}
\end{pmatrix},\ 
\delta
= 
\frac{1}{w\|v\|^2}
\begin{pmatrix}
 0 \\
\displaystyle -\frac{th}{w}\\[2ex]
\displaystyle 
 c_4h
\end{pmatrix}.
\]
Then the first eigenvalue $\tlambda$ and the associated eigenvector $\tv$ of $G(t)$ can be written as
\[
\tlambda = \hw\|v + \delta\|^2 + O(t^8),\ \tv = v + \delta + O(t^8).
\]
\end{lemma}
\begin{proof}
Since $v =(1, O(t), O(t^2))^T$, Lemma $\ref{lemma:rank1}$ implies that
\begin{align*}
 G(t)v = \hw\|v\|^2v + 
\begin{pmatrix}
 0 \\
\displaystyle -\frac{th}{w}\\[2ex]
\displaystyle 
\frac{tg_{23}}{w} + \frac{t^2g_{13}}{w^2}
+ \frac{g_{13}^3}{w^2}
\end{pmatrix} + O(t^8).
\end{align*}
Here, by
$h = {{2t^6}\over{(2c_4 w + 2c_5 t)^4w}}$, we see that
\[
\frac{g_{13}^3}{w^2}
= \frac{t^6}{(2c_4w + 2c_5 t)^3w^2}
= \frac{c_4w + c_5t}{w}h.
\]
In addition, since 
$g_{23} + \frac{tg_{13}}{w} = -c_5h + O(t^7)$ as in $(\ref{eq:G2313})$, we have
\begin{align*}
  \frac{tg_{23}}{w} + \frac{t^2g_{13}}{w^2}
+ \frac{g_{13}^3}{w^2}
 = \frac{t}{w}\left(g_{23} + \frac{tg_{13}}{w}\right) + \frac{g_{13}^3}{w^2}
& = \frac{-c_5th}{w} 
+ \frac{c_4w + c_5t}{w}h + O(t^8)\\
%& =  - \frac{2c_5t^7}{(2c_4 w + 2c_5 t)^4 w^2} 
%+ \frac{t^6}{(2c_4 w + 2 c_5 t)^3 %w^2} + O(t^8) \\
%&= \frac{- 2c_5t^7 + (2c_4 w + 2c_5 %t)t^6}{(2c_4 w + 2c_5 t)^4 w^2} 
%= \frac{2c_4t^6}{(2c_4 w + 2c_5 t)^4 w}  + O(t^8)\\
& = c_4 h + O(t^8).
\end{align*}
Thus we obtain
\[
G(t)v = \hw\|v\|^2v + (0, -\frac{th}{w}, c_4h)^T + O(t^8)
= \hw\|v\|^2v + \hw\|v\|^2\delta + O(t^8).
\]
Since $v = (1,O(t), O(t^2))$, $\delta = (0, O(t^7), O(t^6))^T$ and
$
G(t) = \left(
\begin{smallmatrix}
O(1) & * & * \\
O(t) & O(t^2) & * \\
O(t^2) & O(t^3) & 0
\end{smallmatrix}
\right)
$
, we have
\begin{equation}
   G(t)(v + \delta)  
   = G(t)v + O(t^8) 
%   =  \hw \|v\|^2v +\hw\|v\|^2\delta  +O(t^8)
% = \hw\|v\|^2(v + \delta) + O(t^8) 
= \hw\|v + \delta\|^2(v + \delta) + O(t^8)\label{eq:approxeig}.
\end{equation}
Let $\tlambda(t)$ and $\tv(t)$ be the first eigenvalue and the associated eigenvector of $G(t)$. Then $\tlambda(0) = 1$, and we may assume $\tv(0) = (1,0,0)^T$.
Let $\epsilon = \tlambda - \hw\|v + \delta\|^2$ and
$(0,r_1,r_2)^T = \tv - (v + \delta)$.
Then we have
\[
G(t)(v + \delta + (0,r_1,r_2)^T) = (\hw\|v+\delta\|^2 + \epsilon)(v + \delta + (0,r_1,r_2)^T).
\]
Thus the equation $(\ref{eq:approxeig})$
implies that
\[
\hw\|v + \delta\|^2
\begin{pmatrix}
 0 \\
 r_2 \\
 r_3 
\end{pmatrix}
+ \epsilon (v+\delta) 
+ \epsilon
\begin{pmatrix}
 0 \\
 r_2 \\
 r_3
\end{pmatrix} + O(t^8)
= G(t)
\begin{pmatrix}
 0 \\
 r_2 \\
 r_3 
\end{pmatrix} = 
\begin{pmatrix}
 t r_2 + O(t^2) r_3 \\
 O(t^2) r_2 + O(t^3) r_3 \\
 O(t^3) r_2
\end{pmatrix}.
\]
Since $v+\delta=(1,O(t),O(t^2))$, we have $\epsilon = tr_2 + O(t^2) r_3 + O(t^8)$.
By $\hw\|v + \delta\|^2= 1 + O(t)$, we obtain
\[
\begin{pmatrix}
 r_2 \\
 r_3 
\end{pmatrix}
+
 \begin{pmatrix}
 O(t^2) r_2 + O(t^3) r_3 \\
O(t^3) r_2 + O(t^4) r_3
\end{pmatrix}
+
O(t)
\begin{pmatrix}
 r_2 \\
 r_3 
\end{pmatrix}
 + O(t^8)
= 
\begin{pmatrix}
 O(t^2) r_2 + O(t^3) r_3 \\
 O(t^3) r_2
\end{pmatrix}.
\]
By the second equality, we see that $r_3 = O(t^3) r_2 + O(t^8)$.
Thus the first equality gives 
$r_2 = O(t^8)$. Then we obtain $r_3 = O(t^8)$ and $\epsilon = O(t^8)$.
Therefore, we complete the proof.
\end{proof}

\noindent\textbf{Determinant.}
When we calculate $P_{\S^3_+}(U)$ for $U$ sufficiently close to $U_*$, we have to consider two cases; $P_{\S^3_+}(U)$ is rank $2$ or rank $1$.
Suppose that $\lambda_1 \geq \lambda_2 \geq\lambda_3$ are the eigenvalues of $U$
and $v_1,v_2,v_3$ are the associated eigenvectors of $U$ respectively.
Since $U$ is sufficiently close to $U_*$, we see that $\lambda_1$ is close to $1$ and $\lambda_3<0$.
If $\lambda_1,\lambda_2>0$, then 
$P_{\S^3_+}(U) = \lambda_1 \frac{v_1v_1^T}{\|v_1\|^2} + \lambda_2 \frac{v_2v_2^T}{\|v_2\|^2}$. If $\lambda_1>0$ and $\lambda_2 \leq 0$, then
$P_{\S^3_+}(U) = \lambda_1 \frac{v_1v_1^T}{\|v_1\|^2}$.
Thus the rank of $P_{\S^3_+}(U)$ is determined by the sign of $\det U$.
We show that the determinant of $G(t)$ is positive for $t$ sufficiently close $0$
under higher order perturbations. Then we have that $P_{\S^3_+}(G(t))$ has rank $1$.
%Especially, if $\det U>0$, then $P_{\S^3_+}(U)$ has rank $1$.

\begin{lemma}
\label{lemma:det}
\[
 \det \left(G(t) + 
\begin{pmatrix}
 O(t^7) & O(t^7) & O(t^7) \\
 O(t^7) & O(t^7) & O(t^7) \\
 O(t^7) & O(t^7) & 0
\end{pmatrix}
\right) = \frac{t^{10}}{32 c_4^6} + O(t^{11}).
\] 
In particular, $P_{\S^3_+}(G(t))$ has rank $1$ for $t$ sufficiently close to $0$.
\end{lemma}
\begin{proof}
Recall that $g_{13} = O(t^2),\ g_{23} = O(t^3)$, and that $2c_4g_{13}-2c_5g_{23} = \frac{t^2}{w} - h + O(t^7)$ and
$\frac{tg_{13}}{w} + g_{23} = O(t^6)$ by $(\ref{eq:U22})$ and $(\ref{eq:G2313})$. Then we have
 \begin{align*}
&  \det (G(t) + O(t^7)) \\
& = 
\begin{vmatrix}
\hw + O(t^7) & t + O(t^7)& -g_{13} + O(t^7) \\ 
t + O(t^7) & 2c_4g_{13}-2c_5g_{23} + O(t^7) & g_{23} + O(t^7)\\ 
-g_{13} + O(t^7) & g_{23} + O(t^7) & 0 
\end{vmatrix}\\
& = (-g_{13} + O(t^7))
\begin{vmatrix}
  t + O(t^7) & -g_{13} + O(t^7)\\ 
 \frac{t^2}{w} - h + O(t^7) & g_{23} + O(t^7)
\end{vmatrix}
- (g_{23} + O(t^7))
\begin{vmatrix}
 w + O(t^7) & - g_{13} + O(t^7)\\
t + O(t^7) & g_{23} + O(t^7)
\end{vmatrix}\\
& = 
\begin{vmatrix}
 -t g_{13} - w g_{23} + O(t^{8}) & - g_{13} + O(t^7)\\
-\frac{t^2 g_{13}}{w} + h g_{13} - tg_{23} + O(t^{9})  & g_{23} + O(t^7)
\end{vmatrix}\\
& = \begin{vmatrix}
 -t g_{13} + wg_{23} & - g_{13} + O(t^7)\\
-\frac{t^2 g_{13}}{w} - tg_{23} & g_{23} + O(t^7)
\end{vmatrix}
+ 
\begin{vmatrix}
 O(t^8) & - g_{13} + O(t^7)\\
 h g_{13} + O(t^9) & g_{23} + O(t^7)
\end{vmatrix} \\
& = \begin{vmatrix}
 -w\left(\frac{t g_{13}}{w} + g_{23}\right) & - g_{13} \\
-t\left(\frac{t g_{13}}{w} + g_{23}\right) & g_{23}
\end{vmatrix}
+ 
\begin{vmatrix}
 0 & - g_{13} \\
 h g_{13} & g_{23}
\end{vmatrix} + O(t^{11})\\
& = 
\left(g_{23} + \frac{t}{w}g_{13}\right)(-tg_{13} - wg_{23})
 + hg_{13}^2 + O(t^{11})\\
& = hg_{13}^2 + O(t^{11})  
= \frac{2t^{10}}{(2c_4 w + 2c_5 t)^6 w} + O(t^{11})
= \frac{t^{10}}{32 c_4^6} + O(t^{11}).
 \end{align*}
\end{proof}

%%%%%%%%%%%%%%%%%%%%%%%%%%%%%%%%%%%%%%%%%%%%%%%%%%%%%%%%%%%%
%\begin{proof}
%Note that $g_{13} = O(t^2),\ g_{23} = O(t^3)$.
%By $(\ref{eq:U22})$ and $(\ref{eq:G2313})$, we have
% \begin{align*}
%  \det G(t) & = 
%\begin{vmatrix}
%w + O(t^8) & t & -g_{13}\\ 
%t & 2c_4g_{13}-2c_5g_{23} & g_{23}\\ 
%-g_{13} & g_{23} & 0 
%\end{vmatrix}\\
%& = -g_{13} 
%\begin{vmatrix}
%  t & -g_{13}\\ 
% \frac{t^2}{w} - h + O(t^7) & g_{23} 
%\end{vmatrix}
%- g_{23}
%\begin{vmatrix}
% w + O(t^8) & - g_{13} \\
%t & g_{23}
%\end{vmatrix}\\
%& = 
%\begin{vmatrix}
% -t g_{13} - w g_{23} + O(t^{11}) & - g_{13} \\
%-\frac{t^2 g_{13}}{w} + h g_{13} + O(t^{9}) - tg_{23} & g_{23}
%\end{vmatrix}\\
%& = \begin{vmatrix}
% -w\left(\frac{t g_{13}}{w} + g_{23}\right) & - g_{13} \\
%-t\left(\frac{t g_{13}}{w} + g_{23}\right) & g_{23}
%\end{vmatrix}
%+ 
%\begin{vmatrix}
% 0 & - g_{13} \\
% h g_{13} & g_{23}
%\end{vmatrix} + O(t^{11})\\
%& = 
%(g_{23} + \frac{t}{w}g_{13})(-tg_{13} - wg_{23})
% + hg_{13}^2 + O(t^{11})\\
%& = hg_{13}^2 + O(t^{11})  
%= \frac{2t^{10}}{(2c_4 w + 2c_5 t)^6 w} + O(t^{11})
%= \frac{t^{10}}{32 c_4^6} + O(t^{11}).
% \end{align*}
%\end{proof}
%%%%%%%%%%%%%%%%%%%%%%%%%%%%%%%%%%%%%%%%%%%%%%%%%%%%%%%%%%%%

Finally, we show the Theorem $\ref{thm:ratPS}$ and $\ref{thm:ratPEPS}$.
\begin{proof}
[Proof of Theorem $\ref{thm:ratPS}$]
By Lemma $\ref{lemma:det}$, we have $\det G(t)>0$. Then $P_{\S^3_+}(G(t))$ has rank $1$ and Lemma $\ref{lemma:approxeig}$ implies that
\[
 P_{\S^3_+}(G(t)) = \tlambda \frac{\tv \tv^T}{\|\tv\|^2}
= \hw\|v + \delta\|^2 \frac{(v+\delta)(v+\delta)^T}{\|v+\delta\|^2} + O(t^8)
= \hw(v + \delta)(v + \delta)^T + O(t^8).
\]
Here we have, for $\delta=(0,\delta_2,\delta_3)$,
\begin{align*}
&   (v + \delta)(v + \delta)^T \\
& = v v^T + \delta v^T + v \delta v^T + \delta \delta^T\\
 & =  v v^T + 
\begin{pmatrix}
 0 & 0 & 0 \\
 \delta_2 & 0 & 0 \\
 \delta_3 & t\delta_3 & 0 
\end{pmatrix} 
+ 
\begin{pmatrix}
 0 & \delta_2 & \delta_3 \\
 0 & 0 & t\delta_3 \\
 0 & 0 & 0 
\end{pmatrix} + O(t^8) 
  =  v v^T + 
\begin{pmatrix}
 0 & \delta_2 & \delta_3 \\
 \delta_2 & 0 & t\delta_3 \\
 \delta_3 & t\delta_3 & 0 
\end{pmatrix} + O(t^8).
\end{align*}
 Since $\|v\|^2 = 1 + O(t^2)$, we obtain
\begin{align*}
&  P_{\S^3_+}(G(t)) \\
& = \hw v v^T + w
\begin{pmatrix}
 0 & \delta_2 & \delta_3 \\
 \delta_2 & 0 & t\delta_3 \\
 \delta_3 & t\delta_3 & 0 
\end{pmatrix} + O(t^8)
 = \hw v v^T + 
\begin{pmatrix}
 0 & -\frac{th}{w} & c_4 h\\
 -\frac{th}{w} & 0 & c_4 t h\\
  c_4h & 
c_4 t h & 0 
\end{pmatrix} + O(t^8)
%\begin{pmatrix}
% 0 & \displaystyle -\frac{th}{w} &\displaystyle c_4 h\\
%\displaystyle -\frac{th}{w} & 0 &\displaystyle c_4 t h\\
% \displaystyle c_4h & 
%\displaystyle c_4 t h & 0 
%\end{pmatrix} + O(t^8)
\end{align*}
By Lemma $\ref{lemma:rank1}$, we have
$\hw v v^T = G(t) - 
\begin{pmatrix}  
  0 & 0 & 0 \\
  0 & 
-h + \frac{\langle B_2,B_1\rangle}{8\,c_4^5\,\|B_1\|^2}t^7
& g_{23} + \frac{tg_{13}}{w} \\
  0 & g_{23} + \frac{tg_{13}}{w} & -\frac{g_{13}^2}{w} 
 \end{pmatrix}$ and hence
\[
 P_{\S^3_+}(G(t)) = G(t) +
\begin{pmatrix}
 0 &\displaystyle -\frac{th}{w} 
& c_4h\\[1ex]
\displaystyle -\frac{th}{w} 
& \displaystyle h - \frac{\langle B_2,B_1\rangle}{8\,c_4^5\,\|B_1\|^2}t^7
&\displaystyle - g_{23} - \frac{tg_{13}}{w} +c_4 th \\[2ex]
 c_4 h &\displaystyle -g_{23} - \frac{tg_{13}}{w} + c_4 th & \displaystyle
\frac{g_{13}^2}{w}
\end{pmatrix} + O(t^8) 
\]
Here the equation $(\ref{eq:G23})$ implies
\begin{align*}
 - g_{23} - \frac{tg_{13}}{w} + c_4 h
= - \frac{c_4\langle B_3,B_1\rangle + c_5\langle B_2,B_1\rangle}{8c_4^5\|B_1\|^2}t^7 + c_5 h + O(t^8).
\end{align*}
This gives the equation $(\ref{eq:ratPS})$.
\end{proof}

\begin{proof}[Proof of Theorem $\ref{thm:ratPEPS}$]
Let $P_E(X) = U_* + s_1B_1 + s_2B_2 + s_3 B_3$. 
Then %Since $B_1,B_2,B_3$ are not orthogonal to each other, 
$(s_1,s_2,s_3)$ satisfies
\[
  \begin{pmatrix}
  \|B_1\|^2 & \langle B_1,B_2\rangle & \langle B_1, B_3 \rangle \\
  \langle B_2, B_1 \rangle  & \|B_2\|^2 & \langle B_2, B_3 \rangle \\
  \langle B_3, B_1 \rangle  & \langle B_3,B_2\rangle & \|B_3\|^2
 \end{pmatrix}
\begin{pmatrix}
 s_1\\
s_2\\
s_3
\end{pmatrix}
= 
\begin{pmatrix}
 \langle B_1, X - U_* \rangle \\
 \langle B_2, X - U_* \rangle \\
 \langle B_3, X - U_* \rangle 
\end{pmatrix}.
\]
Let $H$ be the coefficient matrix of the left hand side and 
$D(t) = P_{\S^3_+}(G(t)) - G(t)$.
Then we have
\[
  P_E(X) 
 = 
\begin{pmatrix}
  B_1 & B_2 & B_3
 \end{pmatrix}
 H^{-1}
\begin{pmatrix}
 \langle B_1, X \rangle \\
 \langle B_2, X \rangle \\
 \langle B_3, X \rangle 
\end{pmatrix}
+ U_* - 
\begin{pmatrix}
  B_1 & B_2 & B_3
 \end{pmatrix}
H^{-1}
\begin{pmatrix}
 \langle B_1, U_* \rangle \\
 \langle B_2, U_* \rangle \\
 \langle B_3, U_* \rangle 
\end{pmatrix},
\]
and hence
\begin{align*}
P_E\circ P_{\S^3_+}(G(t))  = P_E((G + D)(t)) 
& = P_E(G(t)) + 
\begin{pmatrix}
  B_1 & B_2 & B_3
 \end{pmatrix} H^{-1}
\begin{pmatrix}
 \langle B_1, D(t) \rangle \\
 \langle B_2, D(t) \rangle \\
 \langle B_3, D(t) \rangle 
\end{pmatrix}\\
& = G(t) + 
\begin{pmatrix}
  B_1 & B_2 & B_3
 \end{pmatrix} H^{-1}
\begin{pmatrix}
 \langle B_1, D(t) \rangle \\
 \langle B_2, D(t) \rangle \\
 \langle B_3, D(t) \rangle 
\end{pmatrix}.
\end{align*}
By Theorem $\ref{thm:ratPS}$, we have
\begin{align*}
  \begin{pmatrix}
 \langle B_1, D(t) \rangle \\
 \langle B_2, D(t) \rangle \\
 \langle B_3, D(t) \rangle 
\end{pmatrix}
 = 
\begin{pmatrix}
\displaystyle -\frac{2th}{w}\\[2ex]
\displaystyle - 2c_4\frac{\langle B_2,B_1\rangle}{8\,c_4^5\,\|B_1\|^2}t^7\\[2ex]
\displaystyle -\frac{\langle B_3,B_1\rangle}{4c_4^4\|B_1\|^2}t^7 
\end{pmatrix}  + O(t^8) 
&  = 
\begin{pmatrix}
\displaystyle -\frac{t^7}{4c_4^4} \\[2ex]
\displaystyle-\frac{\langle B_2,B_1\rangle}{4c_4^4\|B_1\|^2}t^7 \\[2ex]
\displaystyle-\frac{\langle B_3,B_1\rangle}{4c_4^4\|B_1\|^2}t^7 
\end{pmatrix} + O(t^8)\\
= -\frac{t^7}{4c_4^4\|B_1\|^2}
\begin{pmatrix}
 \|B_1\|^2 \\
\langle B_2,B_1 \rangle\\
\langle B_3,B_1 \rangle
\end{pmatrix} + O(t^8).
\end{align*}
By the Cramel's rule, we obtain
\[
P_E\circ P_{\S^3_+}((G(t))) = G(t) -\frac{t^7}{4c_4^4\|B_1\|^2}B_1 + O(t^8)
= G(t) -\frac{t^7}{8c_4^4(2c_1^2 + 1)}B_1 + O(t^8).
\]
\end{proof}

\section{Application to convergence analysis}
\label{sec:rate}
%check
If the singularity degree of the intersection of $\S^n_+$ and a plane is $2$,
then an upper bound for the convergence rate is given as $O(k^{-1/6})$. However, as shown in Example $\ref{ex:linear}$, an upper bound based on the singularity degree is far from tight in general. 

Throughout this section, we consider $\S^3_+$ and a Type $2$ plane $E$ with $c_4\neq 0$.
Using rational formulas, 
we have the following theorem, which 
shows that the upper bound based on the singularity degree  is actually tight for alternating projections for $E$ and $\S^3_+$. 
\begin{thm}
\label{thm:rate}
% Suppose $E$ is Type $2$ and $c_4 \neq 0$. 
For sufficiently small $t>0$, let $U_0 = G(t)$ in $(\ref{eq:slowest})$ be the initial point 
and construct the AP sequence $U_1, U_2, \ldots$ for $E$ and $\S^3_+$.
 Then $\displaystyle\|U_k - U_*\| = \Theta\left(k^{-\frac{1}{6}}\right)$. Moreover,
\[
 \lim_{k\to \infty}\left(\frac{3}{32c_4^4(2c_1^2 + 1)^4}\right)^{\frac{1}{6}}k^{\frac{1}{6}}\|U_k - U_*\| = 1.
\]
\end{thm}

To prove this theorem, we use the following lemmas, which deal with the first eigenvalue and rational expressions for projections of a matrix obtained by perturbing $G(t)$ with terms of degree $7$.

By applying the Gram-Schmidt process to $B_1,B_2,B_3$ in this order,
we obtain an orthogonal basis of $E$ and denote it by $C_1, C_2, C_3$. 
Note that $C_1 = B_1$.
Let $\lambda$ be the first eigenvalue of $G(t)$ and $v(t)$ be the 
the associated eigenvector with $v_1(t) = 1$. By Lemma $\ref{lemma:approxeig}$, we see that
\[
 \lambda(t) = 1 + O(t),\quad 
v(t) = 
\left(1, O(t), O(t^2)\right)^T.
\]
Let $H(t) = \beta t^7 C_2 + \gamma t^7 C_3$ and 
$\text{{\small
$\begin{pmatrix}
 \eta_1(t) & \eta_2(t) & \eta_3(t)\\
\eta_2(t) & \eta_4(t) & \eta_5(t)\\
\eta_3(t) & \eta_5(t) & \eta_6(t)
\end{pmatrix}$}}
= H(t).
$
\begin{lemma}
\label{lemma:perteig}
 Let $\tlambda$ be the first eigenvalue of $(G + H)(t)$ and 
$\tv(t)=(\tv_1(t),\tv_2(t),\tv_3(t))$ be the associated eigenvector with $\tv_1(t)=1$.
Then
\begin{align*}
\tlambda(t)  & = \lambda + \eta_1(t) + O(t^8),\quad 
\tv(t)  = v(t) + 
\begin{pmatrix}
 0 \\
\eta_2(t) + O(t^8)\\
\eta_3(t) + O(t^8)
\end{pmatrix} 
\end{align*}
\end{lemma}
\begin{proof}
 Since $\lambda$ and $\tlambda$ are simple eigenvalues of $G(t)$ and $(G + H)(t)$ respectively, both eigenvalues and associated eigenvectors are analytic functions in $t$. Let $i$th homogeneous parts of $G, \lambda, v$ be $G_i,\lambda_i,v_i$ respectively. Then these are decomposed as
\[
  G = G_0 + G_1 + \cdots,\quad \lambda = \lambda_0 + \lambda_1 +
 \cdots,\quad \tv = v_0 + v_1 + \cdots.
\]
Note that $G_0 = 
\left(
\begin{smallmatrix}
 1 & 0 & 0\\
 0 & 0 & 0\\
 0 & 0 & 0
\end{smallmatrix}
\right),\ 
\lambda_0 = 1,\ v_{0,1} = 1,\ v_{i,1} = 0\ (i = 1, 2, \ldots)$.
Then we have
\[
  (G_0 + G_1 + \cdots)(v_0 + v_1 + \cdots) = (\lambda_0 + \lambda_1 + \cdots)(v_0 + v_1 + \cdots).
\]
%The terms of degree up to $1$ in both side satisfy
%\[
%   A_0v_1 + A_1v_0 = \lambda_0 v_1 + \lambda_1 v_0,\quad
%  \lambda_1 v_0 + (\lambda_0 E - A_0)v_1 = A_1 v_0,\quad
%  \begin{pmatrix}
%    \lambda_1\\
%    v_{1,2} \\
%    v_{1,3}
%   \end{pmatrix} = A_1v_0.
%\]
%Thus $\lambda_1$ and $v_1$ are determined.
Let $I$ be the identity matrix. By comparing the terms of degree $n$, we obtain
 \begin{align}
  G_0 v_n + G_1 v_{n-1} + \cdots + G_n v_{0}
 & = \lambda_0 v_n + \lambda_1 v_{n-1} + \cdots + \lambda_n v_0,\notag \\
 \lambda_n v_0 + (\lambda_0 I - G_0)v_n
 & = G_1 v_{n-1} + \cdots + G_n v_0 - (\lambda_1 v_{n-1} + \cdots
 \lambda_{n-1}v_1), \notag \\
  \begin{pmatrix}
    \lambda_n\\
    v_{n,2} \\
    v_{n,3}
   \end{pmatrix}
 & = G_1 v_{n-1} + \cdots + G_n v_0 - (\lambda_1 v_{n-1} + \cdots \lambda_{n-1}v_1). \label{eq:powereig}
 \end{align}
Since the right hand side of $(\ref{eq:powereig})$ has only the terms of the eigenvalues and eigenvectors 
of degree less than or equal to $n-1$, $\lambda_n$ and $v_n$ are determined iteratively by the lower degree parts and the parts of $G(t)$ of degree less than or equal to $n$.

Next, we consider the eigenpair $\tlambda$ and $\tv$ of $G+H$. 
Since $H$ consists of homogeneous polynomials of degree $7$,
we see that $G+H$ coincides with $G$ up to degree $6$. Thus $\tlambda$ and $\tv$ satisfy the equation $(\ref{eq:powereig})$ for $n \leq 6$ and 
\begin{align*}
  \begin{pmatrix}
    \tlambda_7\\
    \tv_{7,2} \\
    \tv_{7,3}
   \end{pmatrix}
& = G_1 v_{6} + \cdots + (G_7 + H) v_0 - (\lambda_1 v_{6} + \cdots \lambda_{6}v_1)\\
& = 
\begin{pmatrix}
 \lambda_7\\
v_{7,2}\\
v_{7,3}
\end{pmatrix} + H v_0
= \begin{pmatrix}
 \lambda_7 + \eta_1\\
v_{7,2} + \eta_2\\
v_{7,3} + \eta_3
\end{pmatrix}.
\end{align*}
\end{proof}
By using the expressions of the eigenpair of $G+H$, 
the following lemma shows that the projections of $G$ and $G+H$ onto $\S^3_+$ coincide up to degree $6$.
\begin{lemma}
\label{lemma:pertPS3+}
If $G(t),\ (G+H)(t)$ are mapped to rank $1$ matrices by  $P_{\S^3_+}$, then
\[
 P_{\S^3_+}((G+H)(t)) 
 =  P_{\S^3_+}(G(t)) 
+ \begin{pmatrix}
   \eta_1 & \eta_2 & \eta_3 \\
   \eta_2 & 0 & 0 \\
   \eta_3 & 0 & 0
  \end{pmatrix} + O(t^8)
\]
\end{lemma}
\begin{proof}
 Let $\tlambda$ and $\tv$ be the first eigenvalue and the associated eigenvector of $G+H$ with $\tv_1=1$, and $\teta = (0, \eta_2, \eta_3)^T$. 
Since $\lambda = 1 + O(t),\ v = (1,O(t),O(t^2))^T$ and $\eta_1,\eta_2,\eta_3 = O(t^7)$, Lemma $\ref{lemma:perteig}$ implies that
\begin{align*}
&  P_{\S^3_+}((G+H)(t)) 
 =  \tlambda\frac{\tv \tv^T}{\|\tv\|^2} \\
& = (\lambda + \eta_1)\frac{(v + \teta)(v + \teta)^T}{\|v + \teta\|^2 } + O(t^8)
 = (\lambda + \eta_1)\frac{v v^T + \teta v^T + v \teta^T + \teta \teta^T}{\|v + \teta\|^2 } + O(t^8)\\
& = \frac{(\lambda + \eta_1)}{\|v + \teta\|^2 }
\left(v v^T + 
\begin{pmatrix}
 0 & \eta_2 & \eta_3 \\
 \eta_2 & 0 & 0 \\
\eta_3 & 0 & 0 
\end{pmatrix}\right) + O(t^8)\\
%& = \frac{(\lambda + \eta_1)}{\|v + \teta\|^2 }
%v v^T + 
%\begin{pmatrix}
% 0 & \eta_2 & \eta_3 \\
% \eta_2 & 0 & 0 \\
%\eta_3 & 0 & 0 
%\end{pmatrix} + O(t^8)\\
& = \frac{\lambda}{\|v + \teta\|^2 }
v v^T +  + \eta_1 v v^T + 
\begin{pmatrix}
 0 & \eta_2 & \eta_3 \\
 \eta_2 & 0 & 0 \\
\eta_3 & 0 & 0 
\end{pmatrix} + O(t^8)\\
& = \frac{\lambda}{\|v\|^2 }
v v^T +  
\begin{pmatrix}
 \eta_1 & \eta_2 & \eta_3 \\
 \eta_2 & 0 & 0 \\
\eta_3 & 0 & 0 
\end{pmatrix} + O(t^8).
\end{align*}
\end{proof}

Let $\tC_2,\ \tC_3$ be the matrices that are equal to $C_2,\ C_3$ 
except for the first low and column being zero vectors, respectively.
Let $\tP_E(X) = P_E(X) - P_E(O)$. Then $\tP_E$ is the linear part of $P_E$ and we have
\begin{align*}
  P_E\circ P_{\S^3_+}((G+H)(t))
&  = P_E\left(P_{\S^3_+}(G(t)) + 
\begin{pmatrix}
   \eta_1 & \eta_2 & \eta_3 \\
   \eta_2 & 0 & 0 \\
   \eta_3 & 0 & 0 
\end{pmatrix}
+ O(t^8)\right)\\
& = P_E\circ P_{\S^3_+}(G(t))
+ \tP_E
\left(
\begin{pmatrix}
   \eta_1 & \eta_2 & \eta_3 \\
   \eta_2 & 0 & 0 \\
   \eta_3 & 0 & 0 
\end{pmatrix}
\right) + O(t^8).
\end{align*}
A representing matrix for $\tP_E$ is given as follows.
\begin{lemma}
\label{lemma:pertPE}
 \begin{equation}
\tP_E\left(\begin{pmatrix}
   \eta_1 & \eta_2 & \eta_3 \\
   \eta_2 & 0 & 0 \\
   \eta_3 & 0 & 0
  \end{pmatrix}\right)
= 
\begin{pmatrix}
C_2 & C_3
\end{pmatrix}
\begin{pmatrix}
%0 & 0 & 0 \\\\
\displaystyle 1 - \frac{\|\tC_2\|^2}{\|C_2\|^2} 
& \displaystyle -\frac{\langle \tC_2, \tC_3\rangle}{\|C_2\|^2} \\[2ex]
 \displaystyle -\frac{\langle \tC_2, \tC_3\rangle}{\|C_3\|^2}
& \displaystyle 1 - \frac{\|\tC_3\|^2}{\|C_3\|^2}
\end{pmatrix}
\begin{pmatrix}
% 0 \\
\beta t^7 \\
\gamma t^7
\end{pmatrix}.\label{eq:PEH}
\end{equation}
\end{lemma}
\begin{proof}
 Recall that 
$
H(t) = 
 \beta t^7 C_2 + \gamma t^7 C_3 = 
\text{{\scriptsize
$\begin{pmatrix}
 \eta_1 & \eta_2 & \eta_3\\
\eta_2 & \eta_4 & \eta_5\\
\eta_3 & \eta_5 & \eta_6
\end{pmatrix}$}}
$.
Since we can write
$ \text{{\scriptsize
$\begin{pmatrix}
  \eta_1 & \eta_2 & \eta_3 \\
  \eta_2 & 0 & 0 \\
  \eta_3 & 0 & 0
 \end{pmatrix}$}}
= \beta t^7(C_2 - \tC_2) + \gamma t^7 (C_3 - \tC_3)
$,
the orthogonality of $C_1,C_2,C_3$ implies that
\begin{align*}
& \tP_E\left(\begin{pmatrix}
   \eta_1 & \eta_2 & \eta_3 \\
   \eta_2 & 0 & 0 \\
   \eta_3 & 0 & 0
  \end{pmatrix}\right)
 =  \sum_{i=1}^3 \frac{\langle C_i, \beta t^7(C_2 - \tC_2) + \gamma t^7 (C_3 - \tC_3)\rangle}{\|C_i\|^2}C_i\\
 = & \frac{ -\beta t^7 \langle C_1,\tC_2\rangle - \gamma t^7 \langle C_1,\tC_3\rangle}{\|C_1\|^2}C_1
+ \frac{\beta t^7(\|C_2\|^2 - \langle C_2,\tC_2\rangle) 
- \gamma t^7 \langle C_2,\tC_3\rangle}{\|C_2\|^2}C_2\\
& + \frac{-\beta t^7\langle C_3,\tC_2\rangle
+ \gamma t^7(\|C_3\|^2 - \langle C_3,\tC_3\rangle)}{\|C_3\|^2}C_3.
\end{align*}
Since $C_1 = B_1 = \left(
\begin{smallmatrix}
 -2c_1 & 1 & 0 \\
 1 & 0 & 0 \\
 0 & 0 & 0
\end{smallmatrix}
\right)$, we have $\langle C_1,\tC_2\rangle = \langle C_1,\tC_3\rangle = 0$ and thus
\begin{align*}
 \tP_E\left(\begin{pmatrix}
   \eta_1 & \eta_2 & \eta_3 \\
   \eta_2 & 0 & 0 \\
   \eta_3 & 0 & 0
  \end{pmatrix}\right)  & = \left(\left(1 - \frac{\langle C_2,\tC_2\rangle}{\|C_2\|^2}\right) \beta t^7 
- \frac{\langle C_2,\tC_3\rangle}{\|C_2\|^2} \gamma t^7\right) C_2 \\
& + \left(- \frac{\langle C_3,\tC_2\rangle}{\|C_3\|^2}\beta t^7
+ \left(1 - \frac{\langle C_3,\tC_3\rangle}{\|C_3\|^2}\right) \gamma t^7\right) C_3.
\end{align*}
By the definition of $\tC_2$ and $\tC_3$, we have the desired expression.
\end{proof}

Let $c = \frac{1}{4c_4^4\|C_1\|^2}$.
By Theorem $\ref{thm:ratPEPS}$, we have $P_E\circ P_{\S^3_+}(G(t)) = G(t - ct^7) + O(t^8)$ and hence
\begin{align*}
  P_E\circ P_{\S^3_+}((G+H)(t))
& = G(t - ct^7)
+ \tP_E
\left(
\text{{\scriptsize
$\begin{pmatrix}
   \eta_1 & \eta_2 & \eta_3 \\
   \eta_2 & 0 & 0 \\
   \eta_3 & 0 & 0 
\end{pmatrix}$}}
\right) + O(t^8).
\end{align*}
Thus the distance between an $AP$ sequence and the slowest curve
is $N + O(t^8)$, where
\[
N := \left\| \tP_E
\left(
\text{{\scriptsize 
$\begin{pmatrix}
   \eta_1 & \eta_2 & \eta_3 \\
   \eta_2 & 0 & 0 \\
   \eta_3 & 0 & 0 
\end{pmatrix}$}}
\right)\right\|_F.
\]
We will show that $N$ is strictly less than 
\[
  \|H(t)\|_F
= \left\|\beta t^7 C_2 + \gamma t^7 C_3\right\|_F
= \sqrt{(\|C_2\|_F \beta t^7)^2 + (\|C_3\|_F\gamma t^7)^2}.
\]

\begin{lemma}
\label{lemma:pertR}
 Let 
 \[
  R = 
\begin{pmatrix}
 \displaystyle 1 - \frac{\|\tC_2\|^2}{\|C_2\|^2} 
& \displaystyle -\frac{\langle \tC_2, \tC_3\rangle}{\|C_2\|\|C_3\|} \\[2ex]
 \displaystyle -\frac{\langle \tC_2, \tC_3\rangle}{\|C_2\|\|C_3\|}
& \displaystyle 1 - \frac{\|\tC_3\|^2}{\|C_3\|^2}
\end{pmatrix}.
 \]
Then we have $\|R\|_2 < 1$ and
\[
 N \leq \|R\|_2\cdot \|H(t)\|_F.
\]
%Moreover, if $\tC_2, \tC_3$ are linearly independent, then $\|R\|_2 < 1$.
\end{lemma}
\begin{proof}
Let $Q = 
\left(
\begin{smallmatrix}
Q_{11} & Q_{21} \\ Q_{21} & Q_{22}
\end{smallmatrix}
\right)
$ be the matrix that appears in the RHS of $(\ref{eq:PEH})$.
 By Lemma $\ref{lemma:pertPE}$, we have
\begin{align*}
  N^2 & = 
\|
(Q_{11}\beta t^7 +
Q_{21}\gamma t^7) C_2
+ (Q_{21}\beta t^7 +
Q_{22}\gamma t^7) C_3
\|_F^2 \\
& = 
(Q_{11}\beta t^7 + Q_{21}\gamma t^7)^2 \|C_2\|_F^2 + (Q_{21}\beta t^7 + Q_{22}\gamma t^7)^2 \|C_3\|_F^2\\
%= &
%\left(\left(1 - \frac{\|\tC_2\|^2}{\|C_2\|^2}\right) \beta t^7 - \frac{\langle \tC_2,\tC_3\rangle}{\|C_2\|^2}\gamma t^7\right)^2 \|C_2\|^2\\
%& \hspace{3em} + \left(- \frac{\langle \tC_2,\tC_3\rangle}{\|C_3\|^2}\beta t^7 + \left(1 - \frac{\|\tC_3\|^2}{\|C_3\|^2}\right) \gamma t^7 \right)^2\|C_3\|^2\\
%
= &
\left(\left(1 - \frac{\|\tC_2\|^2}{\|C_2\|^2}\right) \|C_2\|\beta t^7 
- \frac{\langle C_2,\tC_3\rangle}{\|C_2\|\|C_3\|}\|C_3\|\gamma t^7\right)^2\\
& \hspace{3em} + \left(- \frac{\langle \tC_2,\tC_3\rangle}{\|C_2\|\|C_3\|}\|C_2\|\beta t^7
+ \left(1 - \frac{\|\tC_3\|^2}{\|C_3\|^2}\right) \|C_3\|\gamma t^7
\right)^2\\
= & \left\|R
\begin{pmatrix}
 \|C_2\|\beta t^7\\
\|C_3\|\gamma t^7
\end{pmatrix}
\right\|_2^2.
\end{align*}
Then we obtain
\[
 N \leq \|R\|_2\cdot\left\|
\begin{pmatrix}
 \|C_2\|\beta t^7\\
\|C_3\|\gamma t^7
\end{pmatrix}
\right\|_2 = \|R\|_2\cdot\|\beta t^7 C_2 + \gamma t^7 C_3\|_F.
\]
Next we estimate $\|R\|_2$. The characteristic polynomial of $R$ is given by
\begin{align*}
  p(\lambda) = &  
\left(\lambda - \left(1 - \frac{\|\tC_2\|^2}{\|C_2\|^2}\right)\right)
\left(\lambda - \left(1 - \frac{\|\tC_3\|^2}{\|C_3\|^2}\right)\right)
- \left(\frac{\langle \tC_2, \tC_3\rangle}{\|C_2\|\|C_3\|}\right)^2\\
=& 
\lambda^2 - \left(2 - \frac{\|\tC_2\|^2}{\|C_2\|^2} - \frac{\|\tC_3\|^2}{\|C_3\|^2}\right)\lambda
+ 1 - \frac{\|\tC_2\|^2}{\|C_2\|^2} - \frac{\|\tC_3\|^2}{\|C_3\|^2}\\
& \hspace{1em}+ \frac{\|\tC_2\|^2\|\tC_3\|^2}{\|C_2\|^2\|C_3\|^2}
- \left(\frac{\langle \tC_2, \tC_3\rangle}{\|C_2\|\|C_3\|}\right)^2.
\end{align*}
Let
\[
 p_1(\lambda) = p(\lambda) + \left(\frac{\langle \tC_2, \tC_3\rangle}{\|C_2\|\|C_3\|}\right)^2,\quad
p_2(\lambda) = p_1(\lambda) - \frac{\|\tC_2\|^2\|\tC_3\|^2}{\|C_2\|^2\|C_3\|^2}.
\]
Recall that $c_4\neq 0$ and $C_1,C_2,C_3$ form an orthogonal basis obtained by applying the Gram-Schmidt process to $B_1, B_2, B_3$ in this order. By the locations of the nonzero elements in $B_i$, we can easily see that $\tC_2$ and $\tC_3$ are linearly independent and hence $|\langle \tC_2, \tC_3\rangle| < \|C_2\|\|C_3\|$. 
Thus we obtain $p_2(\lambda) < p(\lambda) < p_1(\lambda)$. Here, we have
\begin{align*}
& p_1(\lambda) = 0 \Longleftrightarrow 
\lambda = 1 - \frac{\|\tC_2\|^2}{\|C_2\|^2},\ 1 - \frac{\|\tC_3\|^2}{\|C_3\|^2}\\
& p_2(\lambda) = 0 \Longleftrightarrow 
\lambda = 1,\ 1 - \frac{\|\tC_2\|^2}{\|C_2\|^2} - \frac{\|\tC_3\|^2}{\|C_3\|^2}.
\end{align*}
Since $p(\lambda)$ is a convex quadratic function, each solution $\lambda$ to $p(\lambda)=0$ satisfies
\[
 -1 < 1 - \frac{\|\tB_2\|^2}{\|B_2\|^2} - \frac{\|\tB_3\|^2}{\|B_3\|^2} 
< \lambda < 1.
\]
Therefore $\|R\|_2 < 1$.
\end{proof}

Now we can show the following proposition, which means that
if we choose the initial point  sufficiently close to the curve $G(t)$, then the AP sequence moves in the rate of $\Theta(t^7)$ towards the intersection point 
while remaining in a neighborhood of the curve $G(t)$.

%check
\begin{prop}
\label{prop:curve_nbh}
For each $\epsilon>0$, there exists $\delta,\ K>0$ such that if $t, \beta, \gamma$ satisfy $0 < t < \delta$, 
$\|\beta C_2 + \gamma  C_3\|_F < \epsilon$,
then
there exist $\tt, \tbeta, \tgamma$ such that
\begin{align*}
& \|\tbeta C_2 + \tgamma  C_3\|_F < \epsilon,\\ 
& 0 < t - \frac{1}{4c_4^4\|C_1\|^2}t^7 - Kt^8 \leq \tt 
\leq t - \frac{1}{4c_4^4\|C_1\|^2}t^7 + Kt^8 < \delta,\\
&  P_E\circ P_{\S^3_+}\left(G(t) + \beta t^7 C_2 + \gamma t^7 C_3\right) = 
G(\tt)
+ \tbeta \tt^{\,7} C_2 + \tgamma \tt^{\,7} C_3.
\end{align*} 
\end{prop}
\begin{proof}
Let $c = 1/4c_4^4\|C_1\|^2$ and $R$ be the matrix in Lemma $\ref{lemma:pertR}$. 
Note that Lemma $\ref{lemma:det}$ ensures $P_{\S^3_+}(G(t))$ has rank $1$ for sufficiently small $t>0$. Thus $P_{\S^3_+}(G(t))$ is calculated with the first eigenvalue and the associated eigenvector of $G(t)$.
By Lemma $\ref{lemma:pertPS3+}$ and Lemma $\ref{lemma:pertPE}$, we have
\begin{align*}
\tG(t,\beta,\gamma) & := P_E\circ P_{\S^3_+}(G(t) + \beta t^7 C_2 + \gamma t^7 C_3)  =  P_E\circ P_{\S^3_+}((G+H)(t))\\
& = G(t - ct^7)
+ 
\begin{pmatrix}
 C_2 & C_3
\end{pmatrix}
R
\begin{pmatrix}
 \beta t^7\\
 \gamma t^7
\end{pmatrix}
 + O(t^8).
\end{align*}
Here
$\begin{pmatrix}
 C_2 & C_3
\end{pmatrix}
R
\begin{pmatrix}
 \beta \\
 \gamma 
\end{pmatrix} = \beta' C_2 + \gamma' C_3$ for some $\beta', \gamma'\in \R$.
Since Lemma $\ref{lemma:pertR}$ implies that
$\left\|
\begin{pmatrix}
 C_2 & C_3
\end{pmatrix}
R
\begin{pmatrix}
 \beta \\
 \gamma 
\end{pmatrix}
\right\|_F \leq \|R\|_2\|\beta C_2 + \gamma C_3\|_F$ and $\|R\|_2<1$,
we have
 \begin{align}
& \|\beta C_2 + \gamma  C_3\|_F < \epsilon\notag \\
& \Longrightarrow \exists \beta', \gamma' \text{ s.t } 
\|\beta' C_2 + \gamma' C_3\|_F < \|R\|_2\cdot\epsilon,\notag\\ 
& \phantom{\Longrightarrow}  \tG(t)  = 
G\left(t- c t^7\right)
+ t^7 \beta' C_2 +  t^7\gamma' C_3 + O(t^8).\label{eq:approx_proj}
\end{align} 
For 
$U_* = 
\left(
\begin{smallmatrix}
1 & 0 & 0 \\
0 & 0 & 0 \\
0 & 0 & 0 
\end{smallmatrix}
\right)
$, 
let $\phi_C(p) := U_* + p_1C_1 + p_2 C_2 + p_3 C_3$.
Then $\phi_C(p)$ is a diffeomorphism between $\R^3$ to $E$.
Define
$p(t) := \phi_C^{-1}(G(t))$, $\tp(t,\beta,\gamma) := \phi_C^{-1}(\tG(t))$.
%Note that $p(t-ct^7) = \phi_C^{-1}(G(t-ct^7))$.
Since the first eigenvalue is simple, we see that the eigenvalue and the associated eigenvectors of $G(t) + \beta t^7 C_2 + \gamma t^7 C_3$ are analytic in $(t,\beta,\gamma)$ and hence 
$\tG(t,\beta,\gamma) = U_* + \tp_1(t,\beta,\gamma)C_1 + \tp_2(t,\beta,\gamma) C_2 + \tp_3(t,\beta,\gamma) C_3$ is also analytic. Thus, by the Taylor expansion with respect to $t$ about $0$, we can actually rewrite $(\ref{eq:approx_proj})$ as
\begin{align}
  \tG(t) & = G(t - ct^7) + t^7 \beta' C_2 +  t^7\gamma' C_3 
+ t^8 (r_1C_1 + r_2 C_2 + r_3 C_3),\notag\\
& = U_* + \left(p_1(t - ct^7) + r_1t^8\right)C_1 
+ \left(p_2(t - ct^7) + \beta't^7 + r_2t^8 \right)C_2\notag\\
& \phantom{= U_* + \left(p_1(t - ct^7) + r_1t^8\right)C_1 
+ } + \left(p_3(t - ct^7) + \gamma't^7 + r_3t^8 \right)C_3
\label{eq:Gtilde}
\end{align}
where 
\[
 r_i(t,\beta,\gamma) = \frac{1}{8!}\frac{\partial^8 \tp_i}{\partial t^8}(\theta_i t, \beta,\gamma),
\]
for some $\theta_i \in (0,1),\ i = 1,2,3$. Let $r = (r_1,r_2,r_3)$ and
\[
 D_\delta = \{(t,\beta,\gamma)\in \R^3 :0 < t < \delta,\ \|\beta C_2 + \gamma C_3 \|_F < \epsilon \}.
\]
Now we have
\begin{align*}
&  \sup\{\|r(t,\beta,\gamma)\|:(t,\beta,\gamma)\in D_\delta\} \\
& \leq \sup\left\{\left(\sum_{i=1}^3\left|\frac{1}{8!}\frac{\partial^8 \tp_i}{\partial t^8}(t,\beta,\gamma)\right|^2\right)^\frac{1}{2}:(t,\beta,\gamma)\in D_\delta\right\}=: K_\delta.
\end{align*}
%Define $\phi_C(p) = U_* + p_1C_1 + p_2 C_2 + p_3 C_3$ and $p(t-ct^7) = \phi_C^{-1}(G(t-ct^7))$, $\tp(t) = \phi_C^{-1}(\tG(t))$.
First, we show that there exists $\tt$ such that $p_1(t-ct^7) + r_1t^8 = p_1(\tt)$ and $|\tt - (t - ct^7)| \leq 2K_\delta t^8$.
Since $G(t) = U_* + t B_1 + g_2(t) B_2 + g_3(t) B_3 = U_* + p_1(t) C_1 + p_2(t) C_2 + p_3(t) C_3$ and $B_1 = C_1$, we see that
$ p_1(t) = t + \frac{\langle C_1,B_2\rangle}{\|C_1\|^2}g_2(t) 
+ \frac{\langle C_1,B_3\rangle}{\|C_1\|^2}g_3(t) $
and hence $p_1'(0) = 1$. 
By taking $\delta$ smaller if necessary, we may assume that for $0< t < \delta$, $p_1'(t - ct^7) > 1/2$
and that $p_1(t - ct^7) + r_1 t^8$ is in the range of the inverse function $g$ of $p_1$ around $t - ct^7$. Then the Taylor's theorem implies that 
\begin{align*}
 \tt & := g\left(p_1(t - ct^7) + r_1 t^8\right)\\
& = g\left(p_1(t - ct^7)\right) + g'\left(p_1(t - ct^7)\right)r_1t^8
+ \frac{1}{2}g''(t - ct^7 + \theta r_1t^8)(r_1t^8)^2,\\
& = t - ct^7 + \left(g'\left(p_1(t - ct^7)\right)
+ \frac{1}{2}g''(t - ct^7 + \theta r_1t^8)r_1t^8\right)r_1t^8,
\end{align*}
for some $\theta\in (0,1)$.
Since $g'\left(p_1(t - ct^7)\right)< 2$, we have $|t - ct^7 - \tt| \leq 2 K_\delta t^8$.
%%%%%%%%%%%%%%%%%%%%%%%%%%%%%%%%%%%%%%%%%%%%%%%%%%%%%%%%%%%%
Next, the equation $(\ref{eq:Gtilde})$ and $\|\beta'C_2 + \gamma' C_3\|_F < \|R\|_2 \epsilon$ imply that
%\[
%\left\| \tG(t) - G(t - ct^7)  \right\|_F
%\leq 
%\left\|
%t^7 \tbeta C_2 +  t^7\tgamma C_3 
%+ t^8 r(t,\beta,\gamma)
%\right\|_F
%\leq \|R\|_2\epsilon t^7 + C_\delta t^8.
%\]
%Let $\tt = \tG_{12}(t)$. Then we obtain
\begin{align*}
& \|\tp_2(t)C_2 + \tp_3(t)C_3 - (p_2(t - ct^7)C_2 + p_3(t - ct^7)C_3)\|_F \\
& = \|t^7 \beta' C_2 +  t^7\gamma' C_3 
+ t^8 (r_1C_1 + r_2 C_2 + r_3 C_3)\|_F
\leq \|R\|_2\epsilon t^7  + \max_{i\in [3]}\|C_i\|_F K_\delta t^8.
\end{align*}
Let $g(t) = (t, g_2(t), g_3(t))$ be the functions defining the slowest curve $(\ref{eq:slowest})$. 
Recall that $g_2(t) = O(t^2),\ g_3(t) = O(t^3)$.
Since $G(t) = U_* + t B_1 + g_2(t) B_2 + g_3(t) B_3 = U_* + p_1(t) C_1 + p_2(t) C_2 + p_3(t) C_3$ and $B_1 = C_1$, we see that 
$g_2(t)\langle B_2,C_2\rangle 
+ g_3(t)\langle B_3,C_2\rangle= p_2(t)\|C_2\|^2$ and hence $p_2(t) = O(t^2)$. Similarly, $p_3(t) = O(t^2)$.
Thus we have
\begin{align*}
&   \left\|\tp_2(t)C_2 + \tp_3(t)C_3 - (p_2(\tt)C_2 + p_3(\tt)C_3)\right\|_F\\
& \leq \left\|\tp_2(t)C_2 + \tp_3(t)C_3 - (p_2(t - ct^7)C_2 + p_3(t - ct^7)C_3)\right\|_F\\
& \hspace{5em} + \left\|p_2(t - ct^7)C_2 + p_3(t - ct^7)C_3 - (p_2(\tt)C_2 + p_3(\tt)C_3)\right\|_F\\
& \leq \|R\|_2\epsilon t^7 + \max_{i\in [3]}\|C_i\|_F(K_\delta + K'_\delta) t^8,
\end{align*}
for some $K'_\delta = O(\delta)$.
Since $\|R\|_2 <1$, by taking $\delta$ smaller if necessary, we obtain
\begin{align*}
&  0 < t - c t^7 - 2K_\delta t^8 \leq \tt
\leq t - c t^7 + 2K_\delta t^8 < \delta,\\
&  \frac{1}{\tt^7} \left\|(\tp_2(t)C_2 + \tp_3(t)C_3) - (p_2(\tt)C_2 + p_3(\tt)C_3)\right\|_F
< \epsilon,
%\left\|(\tg_2,\tg_3)(t) - (g_2,g_3)(\tt)\right\| < \epsilon t^7
\end{align*}
for all $(t,\beta,\gamma) \in D_\delta$.
The second inequality implies that there exist $\tbeta,\tgamma$ such that
\begin{align*}
&   \|\tbeta C_2 + \tgamma C_3\|_F < \epsilon,\\
& \tp_2(t)C_2 + \tp_3(t)C_3 = 
p_2(\tt)C_2 + p_3(\tt)C_3 + 
\tbeta \tt^7 C_2 + \tgamma \tt^7 C_3.
\end{align*}
Therefore we have
\begin{align*}
  \tG(t) & = U_* + p_1(\tt) C_1 + p_2(\tt)C_2 + p_3(\tt)C_3 + 
\tbeta \tt^7 C_2 + \tgamma \tt^7 C_3\\
& = G(\tt) + \tbeta \tt^7 C_2 + \tgamma \tt^7 C_3.
\end{align*}

\end{proof}

%\begin{prop}
%Let $U_k = \phi(g(t_k))$ for $k=0,1,2,\ldots,$.
%For sufficiently small $\delta>0$, if we choose $t_0, \beta_0, \gamma_0$ as $0 < t_0 < \delta$, $|\beta_0|,\ |\gamma_0| < \delta$, then
%there exist $t_k, \beta_k, \gamma_k$ such that
%\begin{align*}
%& \|\beta B_2 + \gamma  B_3\|_F < \epsilon \\
%& \Longrightarrow \exists \tbeta, \tgamma \text{ s.t } 
%\|\beta' B_2 + \gamma'  B_3\|_F < \|R\|_2\cdot\epsilon,\\ 
%& \phantom{\Longrightarrow} P_E\circ P_{\S^3_+}(\phi(g(t) + \delta(t))) = 
%\phi(g\left(t- c t^7\right)) 
%+ \tbeta t^7 B_2 + \tgamma t^7 B_3 + O(t^8)
%\end{align*} 
%\end{prop}

We use the following lemma on a recursive sequence.
\begin{lemma}
\label{lemma:basic}
 Suppose that the sequence $\{x_k\}$ satisfies $(q+1)C - (q+2)Kx_0>0$ and 
\[
0 < x_{k-1}(1 - C x_{k-1}^q - K x_{k-1}^{q+1})
\leq x_k \leq 
x_{k-1}(1 - C x_{k-1}^q + K x_{k-1}^{q+1})\quad (k = 1,2,\ldots)
\]
for some $C, K >0,\ q\in \N$. Then
\[
 \lim_{k\to \infty} (qC)^{\frac{1}{q}}k^{\frac{1}{q}}x_k = 1.
\]
%If $\{x_k\}$ satisfies the converse inequality, then the corresponding limit inferior also satisfies the converse inequality.
\end{lemma}
\begin{proof}
We show $x_k\to 0$. Suppose $\alpha:=\inf_{k} x_k>0$. 
Let $f(x) = -Cx^{q+1} + Kx^{q+2}$ and $M=(q+1)C/((q+2)K)$.
Then $f'(x) = -(q+1)Cx^{q} + (q+2)Kx^{q+1} = (-(q+1)C + (q+2)Kx)x^{q} < 0$ for $0<x<M$.
Since $0 < x_0 < M < C/K$, we see $x_1 \leq x_0 + f(x_0) < x_0 $, and hence we obtain inductively that
$x_k \leq x_{k-1}  + f(x_{k-1}) < x_{k-1}$ for all $k$. 
Since $f(x)$ is decreasing for $0 < x < M$, 
we see $\alpha \leq x_k \leq x_{k-1} + f(x) < x_{k-1} - C\alpha^{q+1} + K \alpha^{q+2}$.
Thus we have $\alpha < \alpha + C\alpha^{q+1} - K\alpha^{q+2} \leq x_{k-1}$ for all $k$. Therefore, a contradiction occurs and hence $\inf_k x_k = 0$. Since $x_k$ is decreasing, we obtain $x_k \to 0$. Then almost identical arguments in 
 the proof of \cite[Lemma 3.1]{OSW} gives the result.
\end{proof}
%check
\begin{proof}
[Proof of Theorem $\ref{thm:rate}$]
Let $\epsilon>0$.
By Lemma $\ref{lemma:det}$, for each $\beta,\gamma$ with $\|\beta C_2 + \gamma C_3\|_F < \epsilon$, there exists $\delta'>0$ such that for $t\in [0,\delta')$, 
 $\det\left(G(t) + \beta t^7C_2 + \gamma t^7C_3\right)>0$ and hence $P_{\S^3_+}\left(G(t) + \beta t^7C_2 + \gamma t^7C_3\right)$ has rank $1$.
%
%Since $\beta_k,\gamma_k$ are bounded independently of $\delta$, Lemma $\ref{lemma:det}$ ensures that we may assume $\det G(t_k)>0$ by taking $\delta>0$ smaller if necessary. Then $P_{\S^3_+}(G(t_k))$ has rank $1$ for each $k$.
Let $c := \frac{1}{4c_4^4\|C_1\|^2} = \frac{1}{4c_4^4(4c_1^2 + 2)}$.
By iteratively applying Proposition $\ref{prop:curve_nbh}$,
there exists $\delta>0$ such that for $\beta_0=\gamma_0 = 0$ and $t_0$ 
with $0<t_0<\delta$, we can construct 
$(t_k, \beta_k, \gamma_k),\ k=1,2,\ldots$, satisfying
\begin{align*}
& \|\beta_k C_2 + \gamma_k  C_3\|_F < \epsilon,\\ 
& 0 < t_{k-1} - c t_{k-1}^7 - Kt_{k-1}^8 \leq t_k
\leq t_{k-1} - c t_{k-1}^7 + Kt_{k-1}^8 < \delta,\\
&  P_E\circ P_{\S^3_+}\left(G(t_{k-1}) + \beta_{k-1} t_{k-1}^7 C_2 + \gamma_{k-1} t_{k-1}^7 C_3\right) = 
G(t_k)
+ \beta_k t_k^{\,7} C_2 + \gamma_k t_k^{\,7} C_3,
\end{align*} 
for some $K>0$.
Then Lemma $\ref{lemma:basic}$ implies $\lim_{k\to \infty}(6c)^{\frac{1}{6}}k^{\frac{1}{6}}t_k = 1$.
Since $\|U_k - U_*\| = (4c_1^2 + 2)^{\frac{1}{2}}t_k + O(t_k^2)$, we obtain
\begin{align*}
 1  = \lim_{k\to \infty}(6c)^{\frac{1}{6}}k^{\frac{1}{6}}t_k
& = \lim_{k\to \infty}(6c)^{\frac{1}{6}}k^{\frac{1}{6}}\left((4c_1^2 + 2)^{-\frac{1}{2}}\|U_k - U_*\| + O(t_k^2)\right) \\
& = \lim_{k\to \infty}\left(\frac{3}{32c_4^4(2c_1^2 + 1)^4}\right)^{\frac{1}{6}}k^{\frac{1}{6}}\|U_k - U_*\|.%check
\end{align*}
\end{proof}

\noindent\textbf{Numerical experiments.}
Figure $\ref{fig:rate}$ is consistent with our claim that the convergence rate of $\|U_k - U_*\|$ is $\Theta(k^{-1/6})$ in the case that $c_1=c_4=1, c_2=c_3=c_5=0$ as in Example $\ref{ex:moment}$ and the initial point is taken from the slowest curve.
We observe from the right of Figure $\ref{fig:rate}$ that the plot of 
$1/\|U_k - U_*\|^6$ approximately coincides with the line $196.0+0.0098k$.
Hence $\|U_k - U_*\|\approx (196.0+0.0098k)^{-1/6}\approx 2.16k^{-1/6}$ for sufficiently large $k$.
The estimate in Theorem $\ref{thm:rate}$ gives 
$\|U_k - U_*\| \approx (32\cdot 27)^{1/6}k^{-1/6} \approx 3.08k^{-1/6}$.
This discrepancy between the coefficient given in Theorem $\ref{thm:rate}$ and the results of the numerical experiments is likely due to the slow convergence of the limit in the estimate.

\begin{figure}[ht]
\begin{tabular}{cc}
        \begin{minipage}{.45\textwidth}
            \centering
\begin{tikzpicture}[scale=0.7]
\begin{axis}[grid=major, xlabel={$k$},legend entries={$\sqrt{k}\|U_{k}-U_*\|$, $\sqrt[6]{k}\|U_k-U_*\|$}, legend style={at={(axis cs:52000,40)}}]
\addplot[blue] table [x=k, y=s2kn] {data61.dat};
%\addplot[black, very thick, dotted] table [x=k, y=line2] {data61.dat};
\addplot[red] table [x=k, y=s6kn] {data61.dat};
%\addplot[blue, very thick, dotted] table [x=k, y=line1] {data61.dat};
\end{axis}
\end{tikzpicture}
        \end{minipage}
        \begin{minipage}{.45\textwidth}
\centering
\begin{tikzpicture}[scale=0.7]
\begin{axis}[grid=major, xlabel={$k$}, legend entries={$\|U_k-U_*\|^{-6}$, $196.0 + 9.807\times 10^{-3}k$}, legend style={at={(axis cs:50000,150)}}]
\addplot[red] table [x=k, y=1norm6,] {data61.dat};
\addplot[blue, very thick, dotted] table [x=k, y=line1] {data61.dat};
\end{axis}
\end{tikzpicture}
\label{figure:fig61-2}
        \end{minipage}
    \end{tabular}
\caption{The left figure displays the plots of $\sqrt{k}\|U_{k}-U_*\|$ and $\sqrt[6]{k}\|U_k-U_*\|$  in Example $\ref{ex:moment}$ with the initial point on the slowest curve, and the right figure displays the plots of $\|U_k - U_*\|^{-6}$ and the line fitting.}
\label{fig:rate}
\end{figure}

\section{Conclusion}
In this paper, we derived three new analytic formulas for
sequences constructed by the alternating projection method applied to an affine space and the cone of positive semidefinite matrices.
In particular, using the first formula, we presented examples that demonstrate gaps between the actual convergence rates and the upper bounds based on singularity degrees.
The second formula was used to construct the slowest curve for a concrete instance of a $3$-plane. The generalization of the slowest curve
for the parametric family of $3$-planes gives rise to the third formula. 
The third formula was applied to show the tightness of the convergence rate of the alternating projection method when applied to a $3$-plane and $\S^3_+$
whose intersection is a singleton and has singularity degree $2$.

We formulate our results in this paper only for cases where the intersection is a singleton, for simplicity of the argument. However, under certain conditions, the argument can be extended to the non-singleton intersection case, which will be the subject of future study.

\section{Acknowledgments}
The first author was supported by JSPS KAKENHI Grant Number JP17K18726 
and JSPS Grant-in-Aid for Transformative Research Areas (A) (22H05107). 
The second author was supported by JSPS KAKENHI Grant Number JP19K03631 and JP24K06841. 
The third author was supported by JSPS KAKENHI Grant Number JP20K11696, JP24K14843 and ERATO HASUO Metamathematics for Systems Design Project (No.JPMJER1603), JST.

\end{document}